\documentclass[11pt]{article}

\usepackage[utf8]{inputenc}
\usepackage[T1]{fontenc}
\usepackage{scalefnt}
\usepackage{natbib}
\usepackage{amsfonts}
\usepackage{mathtools}
\usepackage{amssymb}
\usepackage{amsmath}
\usepackage{amsthm}
\usepackage{xcolor}
\definecolor{mydarkblue}{rgb}{0,0.08,0.85}
\usepackage[colorlinks=true,bookmarks=true,linkcolor=mydarkblue,urlcolor=mydarkblue,citecolor=mydarkblue,breaklinks=true]{hyperref}
\usepackage{breakcites}
\usepackage{enumitem}
\usepackage{cases}
\usepackage{commath}
\usepackage{algorithm}
\usepackage{algpseudocode}
\usepackage{interval}
\intervalconfig{soft open fences} 
\usepackage{xfrac}
\usepackage{tikz}
\usepackage{tikz-cd}
\usepackage{quiver}
\usetikzlibrary{matrix}
\usepackage{makecell}

\usepackage{sectsty}
\makeatletter\def\@seccntformat#1{\protect\makebox[0pt][r]{\csname the#1\endcsname\hspace{12pt}}}\makeatother

\usepackage[verbose=true,letterpaper]{geometry}
\AtBeginDocument{
  \newgeometry{
    textheight=9in,
    textwidth=6in,
    top=1in,
    headheight=12pt,
    headsep=25pt,
    footskip=30pt
  }
}

\newtheorem{theorem}{Theorem}[section]
\newtheorem{proposition}[theorem]{Proposition}
\newtheorem{definition}[theorem]{Definition}
\newtheorem{lemma}[theorem]{Lemma}
\newtheorem{corollary}[theorem]{Corollary}
\newtheorem{remark}[theorem]{Remark}

\newcommand{\zeros}{0}

\newcommand\mmm{\mkern-1mu}

\newcommand{\reals}{{\mathbb R}}

\DeclareMathOperator*{\argmin}{arg\,min}

\newcommand{\retr}{\mathrm{R}}
\newcommand{\grad}{\nabla}
\newcommand{\hess}{\nabla^2}

\newcommand{\stiefel}{\mathrm{St}}
\newcommand{\grassmann}{\mathrm{Gr}}
\newcommand{\tangent}{\mathrm{T}}
\newcommand{\normal}{\mathrm{N}}

\newcommand{\frob}{\mathrm{F}}
\newcommand{\opnorm}{\mathrm{op}}

\DeclareMathOperator{\diag}{diag}

\DeclareMathOperator{\rank}{rank}

\DeclareMathOperator{\proj}{proj}

\newcommand{\sym}{\mathrm{Sym}}
\newcommand{\sksym}{\mathrm{Skew}}

\DeclarePairedDelimiterX{\inner}[2]{\langle}{\rangle}{#1, #2}
\DeclarePairedDelimiterX{\sinner}[3]{#3\langle}{#3\rangle}{#1, #2}

\newcommand{\D}{\mathrm{D}}
\newcommand{\diff}{\mathop{}\!\mathrm{d}}

\newcommand{\smooth}[1]{{\mathrm{C}^{#1}}}

\newcommand{\sequence}[1]{\{#1\}}
\newcommand{\transpose}{^\top\mmm }

\makeatletter
\newcommand{\TODOF}[1]{\@bsphack\@esphack}
\makeatother

\newcommand\restrold[2]{{
    \left.\kern-\nulldelimiterspace 
      #1 
    \right|_{#2} 
  }}

\makeatletter
\newcommand{\normalscaling}{\bBigg@{0}}

\newcommand{\abitbig}{\bBigg@{1}}
\makeatother

\newcommand\restr[3]{{
    #3.\kern-\nulldelimiterspace 
    #1 
    #3|_{#2} 
  }}

\newcommand{\aref}[1]{\hyperref[#1]{A\ref{#1}}}
\newcommand{\cref}[1]{\hyperref[#1]{C\ref{#1}}}
\newcommand{\burermonteiro}{Burer--Monteiro}
\newcommand{\cauchyschwarz}{Cauchy--Schwarz}
\newcommand{\loja}{\L ojasiewicz}
\newcommand{\polyakloja}{Polyak--\loja}
\newcommand{\pl}{\ensuremath{\text{P\L}}}
\newcommand{\morsebott}{Morse--Bott}

\newcommand{\mb}{\ensuremath{\text{MB}}}

\usepackage{scalerel,stackengine}
\newcommand\pig[1]{\scalerel*[5pt]{\big#1}{%
    \ensurestackMath{\addstackgap[1.5pt]{\big#1}}}}

\newcommand{\lift}{\varphi}

\newcommand{\manifold}{\mathcal{M}}

\newcommand{\nanifold}{\mathcal{N}}
\newcommand{\sanifold}{\mathcal{S}}

\newcommand{\boundedrank}{\mathbb{R}_{\leq r}^{m \times n}}
\newcommand{\fixedrank}{\mathbb{R}_r^{m \times n}}
\newcommand{\fixedranksub}{\lift^{-1}(\reals_r^{m \times n})}

\newcommand{\mfc}{g}

\newcommand{\sfc}{f}

\newcommand{\optimalset}{\mathcal{S}}
\newcommand{\optpoint}{\bar x}

\newcommand{\desing}{\mathcal{M}}

\newcommand{\sfactor}{S}
\newcommand{\embeddingspace}{\mathcal{E}}

\DeclareMathOperator{\sympart}{sym}

\newcommand{\plconstant}{\mu}

\newcommand{\qa}{\xi_0}
\newcommand{\qb}{\xi_1}
\newcommand{\qc}{\xi_2}
\newcommand{\qd}{\xi_3}
\newcommand{\projdiff}{\mathcal P}
\newcommand{\iacc}[1]{#1''}
\newcommand{\acc}[1]{\ddot #1}

\title{Optimization over bounded-rank matrices through a desingularization
  enables joint global and local guarantees}

\author{
  Quentin Rebjock and Nicolas Boumal\thanks{Ecole Polytechnique F\'ed\'erale de
    Lausanne (EPFL), Institute of Mathematics.
    This work was supported by the Swiss State Secretariat for Education,
    Research and Innovation (SERI) under contract number MB22.00027.}
}

\date{\today}

\numberwithin{equation}{section}

\begin{document}

\maketitle

\begin{abstract}
  Convergence guarantees for optimization over bounded-rank matrices are delicate to obtain
  because the feasible set is a nonsmooth and nonconvex algebraic variety.
  Existing techniques include direct optimization over bounded-rank matrices (e.g., projected gradient descent), fixed-rank optimization (over the maximal-rank stratum), and the LR parameterization.
  They all lack either global guarantees
  (the ability to accumulate only at stationary points)
  or fast local convergence
  (e.g., if the limit has non-maximal rank).
  We study a lifted geometry that allows algorithms to enjoy both.

  \cite{khrulkov2018desingularization} parameterize the bounded-rank variety via
  a desingularization to recast the optimization problem onto a smooth manifold.
  Building on their ideas, we develop a Riemannian geometry for this desingularization, also with care for numerical considerations.
  We use it to ensure conditions that, for many standard algorithms, yield global convergence to stationary points with fast local rates.
  On matrix completion tasks, we find that this approach is comparable to others.
\end{abstract}

\section{Introduction}\label{sec:intro}

We aim to minimize a continuously differentiable
function $\sfc \colon \reals^{m \times n} \to \reals$ over the set of
bounded-rank matrices:
\begin{align}\label{eq:problem}\tag{P}
  \min_{X \in \reals^{m \times n}} \sfc(X) && \text{subject to} && \rank X \leq r,
\end{align}
for some $r < \min(m, n)$.
This is a classical problem: see Section~\ref{sec:relatedwork} for a literature review.

The feasible set is a nonconvex, nonsmooth algebraic variety:
\begin{align*}
    \boundedrank = \{X \in \reals^{m \times n} : \rank(X) \leq r \}.
\end{align*}
Such problems are computationally hard in general~\citep{gillis2011low}, but we can still aim for local solutions.
There exist at least three paradigms to do so. 

The first consists of methods that optimize directly over $\boundedrank$.
This notably includes the projected gradient descent (PGD) algorithm~\citep{goldfarb2011convergence,jain2010guaranteed,jain2014iterative} and variants~\citep{schneider2015convergence}.
These are first-order methods with limited local convergence rates (at most linear).

Another approach is to optimize only over the maximal rank stratum
\begin{align*}
    \fixedrank = \{X \in \reals^{m \times n} : \rank(X) = r\},
\end{align*}
thereby ignoring matrices of rank
strictly less than $r$.
That stratum accounts for almost all of $\boundedrank$ and it is a smooth
manifold.
Thus, it is possible to endow it with a Riemannian structure and roll out
off-the-shelf second-order optimization
algorithms~\citep{absil2008optimization,shalit2012online,vandereycken2013low}.
Those algorithms generate sequences of matrices of rank $r$.
These may converge to matrices of rank strictly less than $r$, in
which case they ``fall off'' the manifold:
theory for optimization on manifolds breaks down
in that event, which makes it difficult to provide \emph{a priori} guarantees.
This is likely to happen when the parameter $r$ overestimates
the actual rank of the solution.
Additionally, the Hessian can grow unbounded near lesser rank
matrices~\cite[\S2.3]{vandereycken2013low}.

The third, and perhaps most popular, class of techniques is to smoothly
parameterize the feasible set $\boundedrank$.
Specifically, let $\manifold$ be a smooth manifold and let $\lift \colon \manifold
\to \reals^{m \times n}$ be a smooth map such that $\lift(\manifold) =
\boundedrank$.
Then minimizing $\sfc$ over $\boundedrank$ is equivalent to minimizing the
lifted function $\mfc = \sfc \circ \lift$ over $\manifold$, for which we can
apply standard (Riemannian) optimization algorithms:
this last paradigm is what we explore here.

Unfortunately, even when $\sfc$ has favorable properties, the lifted cost function
$\mfc = \sfc \circ \lift$ may not.
As an example, consider the prominent parameterization $\lift(L, R) = L R^\top$ with
domain $\manifold = \reals^{m \times r} \times \reals^{n \times r}$.
We refer to it as the LR parameterization (see also Appendix~\ref{sec:regularizers}).
It reduces the feasible set to a simple Euclidean space, yet the lifted cost
function $\mfc$ has several drawbacks for theoretical analyses:\footnote{One might also worry about $\mfc$ suffering from spurious critical points introduced by the parameterization: see Section~\ref{subsec:points-of-interest} for reassurances.}
\begin{enumerate}
    \item Unbounded sublevel sets for $\mfc$.\\
            This is because the fibers $\lift^{-1}(X) = \{(L, R) : L R^\top = X\}$ are
            unbounded.
            An indirect effect of this is that $\mfc$ may not have a Lipschitz continuous
            gradient even if $\sfc$ does.
            This hinders global convergence analyses.
    \item No local \polyakloja{} (\pl{}) condition around minimizers of $\mfc$ in general.\\
            Indeed, say~\eqref{eq:problem} has a unique local minimizer $X$.
            Its fiber $\varphi^{-1}(X)$ is a set of local minimizers for $\mfc$. 
            If $\mfc$ is twice continuously differentiable ($\smooth{2}$) and satisfies local \pl{} around minimizers, then those form a locally smooth set~\citep{rebjock2024fast}.
            Yet $\varphi^{-1}(X)$ is a nonsmooth set when $\rank(X) < r$.
            Thus, $\mfc$ lacks local \pl{} even in this simple scenario.
            That hinders fast local convergence guarantees.
\end{enumerate}

Accordingly, a promising direction is to explore other smooth parameterizations
of $\boundedrank$ whose fibers would be smooth and bounded.
As it happens, \cite{khrulkov2018desingularization} proposed one such parameterization.
They call it a \emph{desingularization} (a term that comes from algebraic
geometry).
Their construction relies on the following simple fact: a matrix $X$ of size $m \times n$ has rank at most $r$
if and only if its kernel has dimension at least $n - r$, that is, its kernel
contains a subspace of dimension $n - r$.
The set of subspaces of dimension $s$ in $\reals^n$ is the \emph{Grassmannian},
which we define in terms of orthogonal projectors\footnote{
  Equivalently, we could define the Grassmannian as a quotient (Stiefel manifold by
  orthogonal group): see \citep{bendokat2024grassmann} for a review.
  That is the approach of~\cite{khrulkov2018desingularization}.}
as:
\begin{align*}
  \grassmann(n, s) = \big\{P \in \sym(n) : P^2 = P \text{ and } \rank(P) = s\big\},
\end{align*}
where $\sym(n)$ is the set of real symmetric matrices of size $n$.
In this way, $\grassmann(n, s)$ is an embedded submanifold of $\sym(n)$.
The desingularization of bounded-rank matrices is the parameterization
\begin{align}\label{eq:desing}
  \desing = \big\{(X, P) \in \reals^{m \times n} \times \grassmann(n, n - r) : X P = \zeros\big\} && \text{and} && \lift(X, P) = X.
\end{align}
We confirm in Proposition~\ref{prop:desing-man} of Section~\ref{sec:geometry} that $\desing$ is indeed a smooth manifold of dimension $(m + n - r)r$, which is the same as those of $\boundedrank$ and $\fixedrank$.
Given a pair $(X, P) \in \desing$, the equality $X P = \zeros$ implies that the kernel of $X$ contains a subspace of dimension $n - r$.
As such, the projector $P$ acts as a certificate that $\rank X \leq r$.
\begin{figure}
  \centering
  \begin{tikzcd}
    {\reals^{m \times n} \times \sym(n)} && {\desing} && {\reals^{m \times n}} && \reals
    \arrow[hook', from=1-3, to=1-1]
    \arrow["{\varphi(X, P)\;=\;X}"', from=1-3, to=1-5]
    \arrow["{\mfc \;=\; \sfc \,\circ\, \varphi}", curve={height=-24pt}, from=1-3, to=1-7]
    \arrow["\sfc"', from=1-5, to=1-7]
  \end{tikzcd}
  \caption{The desingularization manifold $\desing$ is embedded in $\embeddingspace =
    \reals^{m \times n} \times \sym(n)$.
    Its image through the parameterization is $\lift(\desing) = \boundedrank$.
    Problem~\eqref{eq:problem} is the minimization of $\sfc$ over
    $\boundedrank$.
    This is executed by minimizing the lifted function $\mfc = \sfc \circ \varphi$.}\label{fig:desing-lift}
\end{figure}
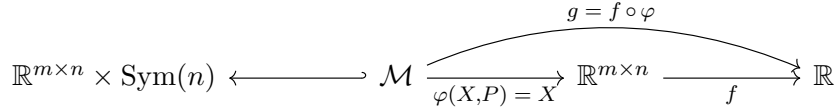
Figure~\ref{fig:desing-lift} illustrates the parameterization of $\boundedrank$
with $\desing$.

\cite{khrulkov2018desingularization} further describe how to compute efficiently with points and tangent vectors of $\desing$, and they build a second-order optimization algorithm that borrows ideas from optimization on manifolds (e.g., retractions) and from constrained optimization (e.g., KKT conditions).
This hybrid method does not come with convergence guarantees.

To exploit their parameterization further,
we endow $\desing$ with a Riemannian metric and with (a few) retractions.
This gives access to a large collection of general-purpose
optimization algorithms, including second-order ones~\citep{absil2008optimization}.
Then, building on work by \citet{levin2023finding},
we exploit the favorable geometry of the fibers of the desingularization
to establish both global and local convergence guarantees for optimization on
$\boundedrank$.
It appears that such guarantees had not been jointly available for a single
algorithm and for a wide class of cost functions.
We conclude with numerical experiments.

\subsection{Contributions}

We develop the geometry of $\desing$ beyond that of~\cite{khrulkov2018desingularization},
derive optimization tools with a focus on numerical efficiency,
secure both global and local convergence guarantees,
and run numerical experiments.
More precisely:
\begin{itemize}
\item In Section~\ref{sec:geometry}, we review the geometry of $\desing$,
  specify parsimonious numerical representations and provide several
  retractions.
  Anticipating our needs for convergence analysis, we study geometric properties of preimages of $\lift$.
\item In Section~\ref{sec:riemann} we equip $\desing$ with a family of Riemannian
  structures.
  We deduce formulas to compute the Riemannian gradient and Hessian of $\mfc =
  \sfc \circ \lift$ in terms of $\sfc$, efficiently.
  We also build second-order retractions.
\item[] Table~\ref{tab:complexities} summarizes the computational complexity of basic operations.
\item We exploit the favorable properties of $\desing$ to obtain
  convergence guarantees:
  \begin{itemize}
  \item In Section~\ref{subsec:global-conv} we show that $\lift$ preserves
    compactness of sublevel sets and also certain Lipschitz properties for the
    cost function.
    We recover global convergence guarantees (accumulation at stationary points of~\eqref{eq:problem}) for a wide range of algorithms, in a similar fashion to~\citep{levin2023finding}.
  \item In Section~\ref{subsec:transitivity} we prove that $\lift$ can preserve
    a type of \polyakloja{} condition around minimizers.
    Assuming convergence to one of those minimizers, we deduce fast local rates for several algorithms, including superlinear rates for second-order algorithms (going beyond~\cite{levin2023finding}).
  \end{itemize}
\end{itemize}
Theorems~\ref{th:global} and~\ref{th:local} outline these guarantees.
Together, they yield algorithms with both global convergence and fast local rates under some assumptions on the cost function.

Finally, in Section~\ref{sec:experiments} we experimentally evaluate
algorithms for matrix completion tasks.
Our implementation is available in the open-source toolbox
Manopt~\citep{boumal2014manopt} and code for the numerical experiments is in a
public GitHub
repository.\footnote{\url{https://github.com/qrebjock/desingularization-experiments}}

\begin{table}[t]
  \centering
  \begin{minipage}[t]{.49\linewidth}
    \begin{tabular}[t]{c|c}
      Operation & Flop count\\
      \noalign{\vskip 1pt}\hline\noalign{\vskip 1pt}
      Retraction \hfill \S\ref{subsec:retractions} & $(m + n)r^2$\\
      \noalign{\vskip 1pt}\hline\noalign{\vskip 1pt}
      Inner product \hfill \S\ref{subsec:metric} & $(m + n)r$\\
      \noalign{\vskip 1pt}\hline\noalign{\vskip 1pt}
      \makecell{Project $(Y, Z)$ onto\\tangent space \hfill \S\ref{subsec:metric}} & $c_{yz} + n r^2$\\
      \noalign{\vskip 1pt}\hline\noalign{\vskip 1pt}
      Gradient \hfill \S\ref{subsec:1derivatives} & $c_g + nr^2$\\
      \noalign{\vskip 1pt}\hline\noalign{\vskip 1pt}
      \makecell{Hessian vector\\product \hfill\phantom{a}} \hfill \S\ref{subsec:2derivatives} & \makecell{$c_g + c_h$\\$+ (m + n)r^2$}
    \end{tabular}
  \end{minipage}
  \begin{minipage}[t]{.49\linewidth}
    \begin{tabular}[t]{c|c}
      Object & Memory\\
      \noalign{\vskip 1pt}\hline\noalign{\vskip 1pt}
      \makecell{Point $(X, P) \in \desing$\\$X = U \Sigma V^\top$\\$P = I - VV^\top$} & $(m + n + r)r$\\
      \noalign{\vskip 1pt}\hline\noalign{\vskip 1pt}
      \makecell{Tangent vector $(\dot X, \dot P)$\\$\dot X = K V^\top + U \Sigma V_p^\top$\\$\dot P = -V_p V^\top - V V_p^\top$}& $(m + n)r$\\
    \end{tabular}
  \end{minipage}
  \caption{Summary of time and space complexity for the main operations and
    objects.
    Sometimes the roles of $m$ and $n$ are asymmetric.
    It can be favorable to transpose rectangular problems when $m \ll n$ to have $n \leq m$.
    Another option is to desingularize the bounded-rank variety with the transposed constraint $X\transpose P = \zeros$, instead of $X P = \zeros$.
    The numbers $c_{yz}, c_g$ and $c_h$ are respectively the costs of computing
    \emph{(i)} $Y\transpose U$ and $Z V$,
    \emph{(ii)} $\grad \sfc(X) A$ and $\grad \sfc(X)\transpose B$ for any two matrices $A \in \reals^{n \times r}, B \in \reals^{m \times r}$, and
    \emph{(iii)} $\hess \sfc(X)[\dot X] V$ and $\hess \sfc(X)[\dot X]\transpose U$.
    These numbers scale at most as $mnr$ or $n^2 r$, but the cost can be much smaller if there is additional structure (for example sparsity).}\label{tab:complexities}
\end{table}

\subsection{Related work}
\label{sec:relatedwork}

\paragraph{Context and applications.}

Low-rank optimization is a classical problem with applications in recommender
systems~\citep{koren2009matrix}, Euclidean
embedding~\citep{singer2010uniqueness},
neuro-imaging~\citep{vounou2010discovering} and image
denoising~\citep{nejati2016denoising}.
Matrix completion~\citep{fazel2002matrix,candes2009exact} is arguably the most
prevalent low-rank optimization problem.
There is extensive literature exploring under which conditions a low-rank
matrix can be recovered from some of its entries~\citep{recht2010guaranteed,
  candes2010power, candes2010matrix, keshavan2010matrixfew}.

There are mainly two approaches to formulate low-rank optimization problems.
The first one is to regularize an unconstrained problem (over $m \times n$ matrices) with a nuclear norm penalty to promote low-rank solutions~\citep{fazel2001rank,srebro2004maximum,bach2008consistency}.
The second one imposes a strict limit on the rank of the matrix, resulting in a constrained problem.
This is the direction we follow.
(To solve the unconstrained problem with nuclear norm penalty, some algorithms exploit the low-rank structure of the solution by also explicitly imposing a rank constraint~\citep{mishra2013low,lee2024accelerating}.)

\paragraph{Algorithms.}

Many existing algorithms for~\eqref{eq:problem} cater to a \emph{specific} cost function for a
given application: see for example~\citep{srebro2003weighted, balzano2010online,
  cai2010singular, park2017non, ge2017no, dragomir2021quartic, kummerle2021scalable,
  bauch2021rank, gao2022riemannian, zilber2022gnmr} for weighted low-rank approximation, matrix
completion and matrix sensing.
Some algorithms are instead \emph{general-purpose}, in the sense that they can
deal with a wide range of cost functions.
We compete with those, some of which we list now.

Projected gradient descent techniques~\citep{jain2014iterative,schneider2015convergence} optimize directly over the singular variety $\boundedrank$: see below.

\cite{shalit2012online} and \citet{vandereycken2013low} propose to optimize over the
fixed-rank manifold $\reals_r^{m \times n}$ to take advantage of Riemannian optimization
tools.
These techniques are efficient in practice but theory breaks when iterates
approach lower-rank strata.

\cite{meyer2011linear} parameterize $\fixedrank$ with
quotient geometries.
\cite{levin2023finding} consider several parameterizations of
$\boundedrank$ (including LR and the desingularization
of~\cite{khrulkov2018desingularization}) and propose a special trust-region
method: we build on this extensively in later sections.

\paragraph{Global convergence guarantees.}\label{par:rw-global}

A method is \emph{globally convergent} if its iterates accumulate only at stationary points.
Surprisingly few algorithms are known to satisfy this guarantee on $\boundedrank$.
We review them here.

\citet[Thm.~5.6(i)]{themelis2018forward}, \cite{pauwels2024generic}, and \cite{olikier2025projected} proved that PGD is globally convergent.
In contrast, the P$^2$GD variant~\cite[Alg.~3]{schneider2015convergence} can generate a sequence converging to a non-stationary point~\cite[\S2.2]{levin2023finding}, \citep{levin2020towards, levin2025effect}.
\cite{olikier2023apocalypse,olikier2025low} refine P$^2$GD to ensure global convergence.

There are also guarantees for optimization through a parameterization $\lift
\colon \manifold \to \reals^{m \times n}$ such that $\lift(\manifold) = \boundedrank$.
In this case, one obstacle is that a critical point of the function $\mfc = \sfc
\circ \lift$ may not map to a stationary point of~\eqref{eq:problem}.
But other correspondences are known (see also
Section~\ref{subsec:points-of-interest}).
Notably, \cite{ha2020equivalence} prove that for the LR parameterization, all \emph{second-order} critical points of $\mfc$ map to stationary points of~\eqref{eq:problem}.
\cite{levin2020towards,levin2023finding} extend this to the desingularization (and more) to obtain a globally convergent trust-region type algorithm.
We build on that to add local guarantees.

\paragraph{Local convergence guarantees.}

Under non-degeneracy assumptions, several algorithms enjoy local convergence guarantees to minimizers of \emph{maximal} rank.
For example, we can instantiate general Riemannian optimization
theorems~\citep[\S4.5, \S6.3, \S7.4.2]{absil2008optimization} for the particular
case of $\reals_r^{m \times n}$~\citep{vandereycken2013low}.
This readily gives local convergence guarantees for standard algorithms
(including gradient descent and Newton's method).
Separately,~\citet[Thm.~3.9]{schneider2015convergence} obtain capture-type results and
convergence rates for P$^2$GD based on \loja{}-type arguments.
\cite{park2018finding} prove local linear convergence to maximal rank points for
an alternating descent algorithm with the LR parameterization.
\cite{olikier2025gauss} derive similar local guarantees for different
alternating rules and step sizes (including line search).

Ensuring fast local convergence to matrices of \emph{non-maximal} rank
is more delicate.
There exist only a couple of results in the literature.
\citet[Thm.~3.10]{schneider2015convergence} obtain convergence rates (at most
linear) for a specific retraction-free gradient method, and these are effective even at
non-maximal rank points.
Also, \cite{luo2024tensor} study a general tensor version of the matrix
sensing problem.
Their setting particularizes to matrix sensing over the constant rank
stratum $\reals_r^{m \times n}$.
This yields local linear convergence for Riemannian gradient descent.
Remarkably, it holds even when the sequence converges to a point of rank less
than $r$, though it is for a specific cost function.

In this paper, we systematize such results to a broader range of cost functions
and algorithms with the desingularization geometry.
It also allows for superlinear rates of convergence.

\paragraph{Positive semidefinite case.}

Another related problem is the minimization of a function over the set of
symmetric positive semidefinite matrices with bounded rank.
One popular approach is the \burermonteiro{} parameterization $\varphi \colon Y
\mapsto YY^\top$ of the feasible set.
Unlike the LR parameterization, the fibers of $\varphi$ are compact, which
is valuable for theoretical analyses.
This makes the positive semidefinite case markedly different~\citep{zhuo2024computational,ding2021rank,zhang2021sharp,zhang2024improved,yalccin2022factorization,zhang2023preconditioned,xu2023thepo,xiong2024how}.

\subsection{Conventions}\label{subsec:notation}

Given a smooth manifold $\nanifold$, we let $\tangent \nanifold$ denote its
tangent bundle, and $\tangent_x \nanifold$ the tangent space at a point $x \in
\nanifold$.

The identity matrix of size $r$ is $I_r$.
The Stiefel manifold is $\stiefel(n, r) = \{V \in \reals^{n \times r} : V\transpose V = I_r\}$.
The orthogonal group is $O(n) = \{Q \in \reals^{n \times n} : Q\transpose Q = I_n\}$.
Given a matrix $V \in \stiefel(n, r)$, we often let $V_\perp \in \stiefel(n, n -
r)$ denote any matrix such that $\big[\begin{matrix} V &
V_\perp \end{matrix}\big]$ is in $O(n)$.
(We also say that $V_\perp$ is an \emph{orthonormal completion} of $V$.)

For a linear operator $\mathcal{L}$ we let $\|\mathcal{L}\|_\opnorm$ denote its
operator norm.
For a matrix $A$ we let $\|A\|_2$ denote its spectral norm and $\|A\|_\frob$ its
Frobenius norm.
The singular values of $A$ are ordered as $\sigma_1(A) \geq \sigma_2(A) \geq \sigma_3(A) \geq \cdots$ so that $\|A\|_2 = \sigma_1(A)$.

For a real-valued function $f$ on a smooth manifold, we say $x$ is a \emph{critical point} if the gradient of $f$ at $x$ is zero.
When minimizing $f$ constrained to a set $S$, we say $x$ is a \emph{stationary point} if $\D f(x)[v] \geq 0$ for all $v$ in the (Bouligand) tangent cone to $S$ at $x$.

\section{Geometry of the desingularization}\label{sec:geometry}

This section explores the geometry of $\desing$:
we rederive its tangent spaces, consider some retractions, and analyze the structure of the fibers of $\lift$.

The dimension of $\desing$ is $(m + n - r)r$, which is much smaller than $m n$
when $r \ll \min(m, n)$.
For efficiency, we should store points and tangent vectors using a number of
floating point numbers close to $\dim \desing$.
We follow an approach based on the singular value decomposition (SVD)
as in~\citep{vandereycken2013low}, with some (but limited) resemblance to the
specifics in~\citep{khrulkov2018desingularization}.
Here $\stiefel$ denotes the Stiefel manifold (details are in
Section~\ref{subsec:notation}).

\begin{lemma}\label{lemma:representation}
  Given $(X, P) \in \desing$, there exist $U \in \stiefel(m, r)$, $V \in
  \stiefel(n, r)$, and a diagonal matrix $\Sigma \in \reals^{r \times r}$ with
  nonnegative, non-increasing diagonal entries such that $X = U \Sigma V^\top$ and $P = I - V V^\top$.
\end{lemma}
\begin{proof}
  The eigenvalues of $P$ are $0$ ($r$ times) and $1$ ($n - r$ times).
  Thus we can choose $W \in \stiefel(n, r)$ such that $P = I - W W^\top$.
  The equality $XP = \zeros$ gives that $X = X W W^\top$.
  Now let $X W = U \Sigma H^\top$ be a thin SVD (in particular, $H \in O(r)$).
  Define $V = W H$ to obtain $X = U \Sigma V^\top$ and $P = I - VV^\top$.
\end{proof}

\begin{definition}\label{def:pt-repr}
  The triplet $(U, \Sigma, V)$ is a \emph{representation} of the point $(X, P)
  \in \desing$ if it satisfies the claim of Lemma~\ref{lemma:representation}.
\end{definition}

Note that this representation is \emph{not} unique (for example when
$X = \zeros$ any $U$ applies).
However, a point $(X, P)$ is stored numerically as a triplet $(U, \Sigma, V)$
and therefore all computations are expressed in terms of that representation
(see also Table~\ref{tab:complexities}).
If $V_\perp \in \stiefel(n, n - r)$ is an orthonormal completion of $V$ then $P
= V_\perp^{} V_\perp^\top$.
We often use this in theoretical analyses, but algorithms never have to actually
build this prohibitively large matrix.

\subsection{Smooth structure and tangent spaces}\label{subsec:smooth}

The set $\desing$ is an embedded submanifold of $\embeddingspace = \reals^{m \times n} \times \sym(n)$, as proven by~\cite{khrulkov2018desingularization}.
We provide another proof below, after a technical lemma.

\begin{lemma}\label{lemma:total-manifold}
  Given $s \in \{0, \dots, n\}$, the set $\sanifold = \big\{(X, W) \in \reals^{m
    \times n} \times \stiefel(n, s) : X W = \zeros\big\}$ is a smooth embedded
  submanifold of $\reals^{m \times n} \times \stiefel(n, s)$ of dimension $m(n -
  s) + ns - \dim \sym(s)$.
\end{lemma}
\begin{proof}
  Define the smooth function $h \colon \reals^{m \times n} \times \stiefel(n, s)
  \to \reals^{m \times s}$ as $h(X, W) = XW$.
  This is a local defining function for $\sanifold$.
  Indeed, $h(X, W) = \zeros$ holds if and only if $(X, W) \in \sanifold$.
  Moreover, the differential of $h$ is $\D h(X, W)[\dot X, \dot W] = \dot X W + X \dot W$.
  It is full-rank because for all $Y \in \reals^{m \times s}$, there exists $(\dot X, \dot W)$ such that $\D h(X, W)[\dot X, \dot W] = Y$, e.g., let $(\dot X, \dot W) = (Y W^\top, \zeros)$.
  It follows that $\sanifold$ is an embedded submanifold of dimension $mn + \dim
  \stiefel(n, s) - ms$, or equivalently $mn + ns - \dim \sym(s) - ms$.
\end{proof}

\begin{proposition}\label{prop:desing-man}
  The set $\desing$ is an embedded submanifold of $\embeddingspace$ of dimension $(m + n - r)r$.
\end{proposition}
\begin{proof}
  Lemma~\ref{lemma:total-manifold} with $s = n - r$ implies that $\nanifold = \big\{ (X, W) \in \reals^{m \times n} \times \stiefel(n, n - r) : XW = \zeros \big\}$ is a smooth embedded submanifold of $\reals^{m \times n} \times \stiefel(n, n - r)$ of dimension $mr + n(n - r) - \dim \sym(n - r)$.
  The orthogonal group $O(n - r)$ acts as $\theta((X, W), Q) =
  (X, WQ)$: this is smooth, free and proper.
  Thus the quotient set $\nanifold / O(n - r)$ is a quotient manifold of
  dimension $(m + n - r)r$~\cite[Thm.~21.10]{lee2013smooth}.
  Define $\psi \colon \nanifold / O(n - r) \to \reals^{m \times n} \times \grassmann(n, n - r)$ as $\psi([X, W]) = (X, W W^\top)$.
  The function $\psi$ is a smooth immersion and it is also a homeomorphism onto
  $\desing = \psi(\nanifold / O(n - r))$.
  So $\psi$ is a smooth embedding and we conclude that $\desing$ is an embedded
  submanifold of $\reals^{m \times n} \times \grassmann(n,n-r)$~\cite[Prop.~5.2]{lee2013smooth}.
  Since $\grassmann(n,n-r)$ is embedded in $\sym(n)$, it follows that $\desing$ is also an embedded submanifold of $\embeddingspace$.
\end{proof}

It is easy to see that $\desing$ is closed and
connected since $\grassmann(n, n - r)$ is so.
These properties are desirable for algorithmic convergence.
In contrast, $\reals_r^{m \times n}$ is not closed in $\reals^{m \times n}$.

\begin{proposition}\label{prop:desing-closed}
  The manifold $\desing$ is closed in $\embeddingspace$ and path-connected.
\end{proposition}
    %
    %
  %
  %

We now derive the tangent spaces of $\desing$.
Given a point $(X, P)$, we find a parameterization of the tangent space at $(X,
P)$ that relies on its representation (see Definition~\ref{def:pt-repr}).

\begin{proposition}\label{prop:tangent-space}
  Represent $(X, P) \in \desing$ with $(U, \Sigma, V)$.
  The tangent space at $(X, P)$ is
  \begin{equation}\label{eq:tangent-space}
    \begin{aligned}
      \tangent_{(X, P)} \desing &= \Big\{ \big(K V^\top + U \Sigma V_p^\top, -V_pV^\top - VV_p^\top\big) : K \in \reals^{m \times r},\\&\qquad\qquad\qquad\qquad\qquad\qquad\qquad\qquad\quad\; V_p \in \reals^{n \times r}, V\transpose V_p = \zeros \Big\}.
    \end{aligned}
  \end{equation}
\end{proposition}
\begin{proof}
  Let $V_\perp$ be an orthonormal completion of $V$.
  Given $K \in \reals^{m \times r}$ and $L \in \reals^{(n - r) \times r}$, there
  exist two smooth curves $V \colon \reals \to \stiefel(n, r)$ and $Y \colon
  \reals \to \reals^{m \times r}$ such that
  \begin{align*}
    Y(0) = U\Sigma, && V(0) = V, && Y'(0) = K && \text{and} && V'(0) = V_\perp L.
  \end{align*}
  (See~\cite[\S7.3]{boumal2020introduction} for the curve on Stiefel.)
  Define $X(t) = Y(t)V(t)^\top$ and $P(t) = I - V(t)V(t)^\top$ so that $c(t) =
  (X(t), P(t))$ is a curve on $\desing$.
  We find that
  \begin{align*}
    X'(0) & = KV^\top + U \Sigma L^\top V_\perp^\top && \text{and} &
    P'(0) & = -VL^\top V_\perp^\top - V_\perp L V^\top.
  \end{align*}
  We deduce that $(KV^\top + U \Sigma L_{}\transpose V_\perp^\top, -V_\perp L V^\top - V
  L_{}\transpose V_\perp^\top)$ is in $\tangent_{(X, P)}\desing$.
  The variables $K$ and $V_p = V_\perp L$ have $mr$ and $(n - r)r$ degrees of freedom
  respectively.
  So this spans a subspace of dimension $(m + n - r)r$, which is the dimension of
  $\tangent_{(X, P)}\desing$ (see Proposition~\ref{prop:desing-man}).
\end{proof}

The expression given in Proposition~\ref{prop:tangent-space} reveals a
parsimonious representation for tangent vectors.
We can represent each with a pair $(K, V_p) \in \reals^{m \times r} \times
\reals^{n \times r}$, requiring only $(m + n)r$ real numbers.

\begin{definition}\label{def:tgt-repr}
  Represent $(X, P) \in \desing$ with $(U, \Sigma, V)$.
  A pair $(K, V_p) \in \reals^{m \times r} \times \reals^{n \times r}$ with
  $V_p\transpose V = \zeros$ is a \emph{representation} of
  the tangent vector $(\dot X, \dot P) \in \tangent_{(X, P)} \desing$ if
  \begin{align*}
      \dot X = K V^\top + U \Sigma V_p^\top && \textrm{ and } &&
      \dot P = -V_p V^\top - V V_p^\top.
  \end{align*}
\end{definition}

Given $(U, \Sigma, V)$, this representation of $(\dot X, \dot P)$ is unique.
We express (most) tangent vectors with this representation for numerical
efficiency.
For theoretical analyses we sometimes decompose $V_p = V_\perp L$ for some
matrix $L \in \reals^{(n - r) \times r}$.

\subsection{Retractions}\label{subsec:retractions}

Let $\tangent \desing$ denote the tangent bundle of $\desing$.
A retraction is a smooth map $\retr \colon \tangent\desing \to \desing$ to
navigate on $\desing$.
Given a point on $\desing$ and a tangent direction at that point, it provides a new point on $\desing$.
We let $\retr_{(X, P)} \colon \tangent_{(X, P)}\desing \to \desing$ denote the
restriction of a retraction $\retr$ at $(X, P) \in \desing$.
See~\cite[\S4.1]{absil2008optimization} and~\cite[\S3.6]{boumal2020introduction}
for background and references.
\begin{definition}\label{def:retr}
  A \emph{retraction} is a smooth map $\retr \colon \tangent\desing \to \desing$
  such that each curve defined as $c(t) = \retr_{(X, P)}(t \dot X, t \dot P)$
  satisfies $c(0) = (X, P)$ and $c'(0) = (\dot X, \dot P)$.
\end{definition}

Many retractions are already known for $\reals_r^{m \times n}$~\citep{absil2015low}.
\cite{khrulkov2018desingularization} proposed a retraction for $\desing$ based
on the Q-factor retraction on $\grassmann(n, n - r)$.
(This is one popular retraction for the Grassmannian.) 
We show below that any retraction on $\grassmann(n, n - r)$ yields a retraction
on $\desing$ in the same way.
(Later, in Section~\ref{subsec:2-retractions}, we also consider second-order retractions.)
We express all computations in terms of the parsimonious representations given
in Definitions~\ref{def:pt-repr} and~\ref{def:tgt-repr}.

\begin{proposition}\label{prop:retr-iff}
  Let $\retr^\grassmann \colon \tangent \grassmann(n, n - r) \to \grassmann(n, n
  - r)$ be smooth and define $\retr \colon \tangent \desing \to \desing$ as
  \begin{align*}
    \retr_{(X, P)}(\dot X, \dot P) = \big((X + \dot X)(I - \bar P), \bar P\big), && \text{where} && \bar P = \retr^\grassmann_P(\dot P).
  \end{align*}
  Then $\retr$ is a retraction on $\desing$ if and only if $\retr^\grassmann$ is
  a retraction on $\grassmann(n, n - r)$.
\end{proposition}
\begin{proof}
  Suppose $\retr^\grassmann$ is a retraction. 
  Let $(X, P) \in \desing$ and $(\dot X, \dot P) \in \tangent_{(X, P)}
  \desing$ have representations $(U, \Sigma, V)$ and $(K, V_p)$.
  Let $P(t) = \retr^\grassmann_P(t \dot P)$ and define $c(t) = (X(t), P(t))$
  where $X(t) = (X + t\dot X)(I - P(t))$.
  By design, $c(0) = (X, P)$.
  Moreover, $X'(t) = \dot X(I - P(t)) - (X + t\dot X)P'(t)$ so that
  \begin{align*}
    X'(0) &= \big(K V^\top + U \Sigma V_p^\top\big)VV^\top + U \Sigma V\transpose \big(V_pV^\top + VV_p^\top\big) = KV^\top + U \Sigma V_p^\top.
  \end{align*}
  It follows that $X'(0) = \dot X$ and we also have $P'(0) = \dot P$, so $c'(0)
  = (\dot X, \dot P)$ and $\retr$ is a retraction.
  Conversely, suppose that $\retr$ is a retraction.
  Let $P \in \grassmann(n, n - r)$ and $\dot P \in \tangent_P \grassmann(n, n -
  r)$.
  The point $(\zeros, P)$ is in $\desing$ and if we apply $\retr$ to the tangent
  vector $(\zeros, \dot P)$ we find $\retr_{(\zeros, P)}(\zeros, \dot P) =
  \big(\zeros, \retr_P^\grassmann(\dot P)\big)$, showing that $\retr^\grassmann$
  is also a retraction.
\end{proof}

We now specialize Proposition~\ref{prop:retr-iff} to the Q-factor retraction and
briefly demonstrate how to compute it (this can be compared to~\cite[\S3.4]{khrulkov2018desingularization}).
Let $(X, P) \in \desing$ and $(\dot X, \dot P) \in \tangent_{(X, P)}\desing$ have representations $(U, \Sigma, V)$ and $(K, V_p)$.
Let $V + V_p = QR$ be a thin QR decomposition, which can be computed in
$O(nr^2)$ operations.
Then the retracted point is $(\bar X, \bar P)$ where we define $\bar X = (X +
\dot X) Q Q^\top$ and $\bar P = I - Q Q^\top$.
We can efficiently obtain a representation for it.
First, form the matrix $W = (X + \dot X) Q = (U \Sigma + K) V\transpose Q + (U
\Sigma) V_p\transpose Q$.
This can be done in $O((m + n)r^2)$ operations using the representations of $X$ and $\dot X$.
Now let $W = \bar U \bar \Sigma H^\top \in \reals^{m \times r}$ be a thin SVD,
which can be computed in $O(mr^2)$ operations.
Then $(\bar X, \bar P)$ has representation $(\bar U, \bar \Sigma, \bar V)$ where
$\bar V = Q H$.

\subsection{Geometry of fibers and preimages}\label{subsec:geometry-fibers}

The fiber of a point $X \in \boundedrank$ is the set $\lift^{-1}(X)$.
In this section, we study more generally the preimages $\lift^{-1}(\sanifold)$
of subsets $\sanifold \subseteq \boundedrank$.
Their geometries affect both global and local convergence properties of
optimization algorithms, as we explain in Section~\ref{sec:conv-guarantees}.
A first notable property is that the preimages of compact sets are compact.

\begin{proposition}[{\cite[Ex.~3.13]{levin2023finding}}]\label{prop:lift-proper}
  The map $\lift$ is (topologically) proper.
\end{proposition}
\begin{proof}
    If $\sanifold$ is compact, then $\varphi^{-1}(\sanifold)$ is closed and it
    is included in $\sanifold \times \grassmann(n, n-r)$.
    Since the latter is compact in $\embeddingspace$, this shows that
    $\varphi^{-1}(\sanifold)$ is compact in $\desing$.
\end{proof}

We now seek to understand when $\lift^{-1}(\sanifold)$ is an embedded submanifold of $\desing$.
This will be important to secure the \polyakloja{} condition in Section~\ref{subsec:transitivity}.
We start with the preimages of constant rank strata.

\begin{lemma}\label{lemma:stratum-preimage}
  For all $s \in \{0, \dots, r\}$, the preimage $\lift^{-1}(\reals_s^{m \times n})
  = \big\{(X, P) \in \desing : \rank X = s\big\}$ is a smooth embedded
  submanifold of $\desing$ of dimension $(m + n - s)s + (n - r)(r - s)$.
\end{lemma}
\begin{proof}
  Lemma~\ref{lemma:total-manifold} (with $n - s$ instead of $s$) implies that
  \begin{align*}
    \sanifold = \big\{ (X, W) \in \reals^{m \times n} \times \stiefel(n, n - s) : X W = \zeros \big\}
  \end{align*}
  is a smooth embedded submanifold of $\reals^{m \times n} \times \stiefel(n, n - s)$ of dimension $ms + n (n - s) - \dim \sym(n - s)$.
  Define the Lie group $\mathcal{G}$ and the right action $\theta \colon
  \sanifold \times \mathcal{G} \to \sanifold$ as
  \begin{align*}
    \mathcal{G} = \bigg\{
    \begin{bmatrix}
      Q_1 & \\ & Q_2
    \end{bmatrix}
    : Q_1 \in O(n - r), Q_2 \in O(r - s)\bigg\} && \text{and} && \theta(X, W, Q) = (X, W Q).
  \end{align*}
  This action is smooth, proper and free.
  Hence the set $\mathcal{H} = \sanifold / \mathcal{G}$ is a quotient
  manifold~\cite[Thm.~21.10]{lee2013smooth} of dimension $\dim \sanifold - \dim
  \mathcal{G} = (m + n - s)s + (n - r)(r - s)$.
  The map $\phi \colon \mathcal{H} \to \reals$ defined by $\phi([X, W]) =
  \sigma_s(X)$ ($s$th largest singular value) is continuous, so
  \begin{align*}
    \mathcal{H}_s = \phi^{-1}(\interval[open]{0}{+\infty}) = \big\{[X, W] \in \mathcal{H} : \rank X = s\big\}
  \end{align*}
  is an open submanifold of $\mathcal{H}$.
  Now partition each $W \in \stiefel(n, n - s)$ into $(W_1, W_2)$ with $n - r$
  and $r - s$ columns.
  Define the smooth map $\psi \colon \mathcal{H}_s \to \desing$ as $\psi([X,
  W_1, W_2]) = (X, W_1^{}W_1^\top)$.
  The image of $\psi$ is exactly $\lift^{-1}(\reals_s^{m \times n})$.
  For all $[X, W_1, W_2] \in \mathcal{H}_s$ the kernel of $X$ is the image of $W_1^{}
  W_1^\top + W_2^{} W_2^\top$ because $\rank X = s$.
  Let $X^\dagger$ denote the pseudo-inverse of $X$.
  Since $I_n - X^\dagger X$ is the orthogonal projector onto $\ker X$, it follows that
  \begin{equation}\label{eq:w2w2t}
    W_2^{} W_2^\top = I - X^\dagger X - W_1^{}W_1^\top.
  \end{equation}
  In particular, the equivalence class $[X, W_1, W_2]$ is entirely determined by $(X, W_1^{} W_1^\top)$.
  It follows that $\psi$ is invertible.
  One can further check that $\psi^{-1}$ is smooth using the smoothness of $X
  \mapsto X^\dagger$ over constant-rank
  strata~\cite[\S4]{golub1973differentiation} and local
  sections~\cite[Thm.~4.26]{lee2013smooth}.
  We also find that $\D \psi$ is injective at each point of $\mathcal{H}_s$.
  To see this, define $\bar \psi \colon \mathcal{S} \to \desing$ as $\bar \psi(X, W_1, W_2) = (X, W_1^{}W_1^\top)$.
  This smooth map coincides with $\psi$ through the quotient (it is constant on equivalent classes of $\mathcal{S}$).
  Its differential $\D \bar \psi \colon \tangent_{(X, W)} \mathcal{S} \to
  \tangent_{(X, W_1^{}W_1^\top)}\desing$ satisfies
  \begin{align}\label{eq:diffbarpsi}
    \D \bar \psi(X, W_1, W_2)[\dot X, \dot W_1, \dot W_2] = (\dot X, \dot W_1^{} W_1^\top + W_1^{} \dot W_1^\top).
  \end{align}
  Suppose this quantity is zero.
  This readily gives $\dot X = \zeros$.
  Since $\dot W_1$ and $\dot W_2$ are tangent vectors of Stiefel manifolds,
  there exist~\cite[\S7.3]{boumal2020introduction} skew-symmetric matrices
  $\Omega_1 \in \sksym(n - r), \Omega_2 \in \sksym(r - s)$ and $B_1 \in
  \reals^{r \times (n - r)}, B_2 \in \reals^{(n - r + s) \times (r - s)}$ such that
  \begin{align*}
    \dot W_1 = W_1 \Omega_1 + W_{1, \perp} B_1 && \text{and} && \dot W_2 = W_2 \Omega_2 + W_{2, \perp} B_2,
  \end{align*}
  where $W_{1, \perp}$ and $W_{2, \perp}$ are orthonormal completions of $W_1$
  and $W_2$.
  The differential~\eqref{eq:diffbarpsi} is zero, so in particular $W_{1,\perp}^\top \dot W_1 = \zeros$, and we deduce that $B_1 = \zeros$.
  Furthermore, the identity~\eqref{eq:w2w2t} implies a link between $\dot X$, $\dot W_1$ and $\dot W_2$.
  Specifically, differentiating~\eqref{eq:w2w2t} gives
  \begin{equation*}
    \dot W_2^{} W_2^\top + W_2^{} \dot W_2^\top = -\dot Z X - Z \dot X - \dot W_1^{} W_1^\top - W_1^{} \dot W_1^\top,
  \end{equation*}
  where $Z = X^\dagger$, and $\dot Z$ is its differential along $\dot X$.
  Multiply this equality by $W_{2, \perp}^\top$ on the left and by $W_2$ on the right to get $W_{2, \perp}^\top \dot W_2 = 0$, showing that $B_2 = 0$.
  Therefore, $(\dot X, \dot W_1, \dot W_2) = (\zeros, W_1\Omega_1, W_2\Omega_2)$, i.e., it lies in the vertical space of $\mathcal{S}$ at $(X, W)$~\cite[\S9.4]{boumal2020introduction}.
  Consequently, $\ker \D \bar \psi(X, W_1, W_2)$ equals this vertical space, and hence $\D \psi$ is injective~\cite[\S3.5.8]{absil2008optimization}.
  (This is because vertical directions in the total space project to the zero tangent vector in the quotient.)
  The map $\psi$ is hence an immersion and a homeomorphism onto its image
  $\lift^{-1}(\reals_s^{m \times n})$.
  It is thus an embedding and its image is a smooth embedded
  submanifold of $\desing$~\cite[Prop.~5.2]{lee2013smooth}, of the same
  dimension as $\mathcal{H}$.
\end{proof}

We now generalize this result to embedded submanifolds of $\reals^{m \times n}$
included in a constant rank stratum, and also give an expression for the tangent
spaces.

\begin{proposition}\label{prop:submanifold-preimage}
  Let $\sanifold$ be an embedded submanifold of $\reals_s^{m \times n}$.
  Then $\lift^{-1}(\sanifold)$ is an embedded submanifold of $\desing$ of
  dimension $\dim \sanifold + (n - r)(r - s)$.
  If $\sanifold$ is $\smooth{k}$ for some $k$ then $\lift^{-1}(\sanifold)$ is
  also $\smooth{k}$.
  Moreover, the tangent space is
  \begin{align*}
    \tangent_{(X, P)} \lift^{-1}(\sanifold) = \big\{ (\dot X, \dot P) \in \tangent_{(X, P)} \desing : \dot X \in \tangent_X \sanifold \big\}.
  \end{align*}
\end{proposition}
\begin{proof}
  The set $\desing_s = \{(X, P) \in \desing : \rank X = s\}$ is a smooth
  embedded submanifold of $\desing$ (Lemma~\ref{lemma:stratum-preimage}).
  We show that $\lift^{-1}(\sanifold)$ is an embedded submanifold of $\desing_s$
  and the first claim follows by transitivity.
  Let $\bar X \in \sanifold$ and $d = \dim \sanifold$.
  The dimension of $\reals_s^{m \times n}$ is $(m + n - s)s$.
  So there exists a neighborhood $\mathcal{U}$ of $\bar X$ in $\reals_s^{m
    \times n}$ and a local defining function $h \colon \mathcal{U} \to
  \reals^{(m + n - s)s - d}$ of $\sanifold$ around $\bar X$.
  In particular $\ker \D h(\bar X) = \tangent_{\bar X} \sanifold$.
  Let $\bar P$ be such that $(\bar X, \bar P) \in \lift^{-1}(\bar X)$ and
  $\mathcal{V} = \lift^{-1}(\mathcal{U})$ be a neighborhood of $(\bar X, \bar
  P)$ in $\desing_s$.
  Define the map $H \colon \mathcal{V} \to \reals^{(m + n - s)s - d}$ as $H(X,
  P) = h(X)$.
  We show that $H$ is a local defining function for $\lift^{-1}(\sanifold)$
  around $(\bar X, \bar P)$.
  For all $(X, P) \in \mathcal{V}$ the equality $H(X, P) = \zeros$ holds if and
  only if $(X, P) \in \lift^{-1}(\sanifold)$.
  The differential $\D H(\bar X, \bar P) \colon \tangent_{(\bar X, \bar P)}
  \desing_s \to \reals^{(m + n - s)s - d}$ is $\D H(\bar X, \bar P)[\dot X, \dot
  P] = \D h(\bar X)[\dot X]$ and we now prove that it is surjective.
  For this we use the surjectivity of $\D h(\bar X) \colon \tangent_{\bar X}
  \reals_s^{m \times n} \to \reals^{(m + n - s)s - d}$ (recall indeed that $\ker
  \D h(\bar X) = \tangent_{\bar X} \sanifold$).
  Given $\dot X \in \tangent_{\bar X} \reals_s^{m \times n}$, there exists $\dot
  P$ such that $(\dot X, \dot P) \in \tangent_{(\bar X, \bar P)} \desing_s$.
  This is because if $X \colon \reals \to \reals_s^{m \times n}$ is a smooth
  curve with $X(0) = \bar X$ and $X'(0) = \dot X$ then we can construct a smooth
  curve $c \colon \reals \to \desing_s$ such that $c(0) = (\bar X, \bar P)$ and
  $c'(0) = (\dot X, \dot P)$ for some $\dot P$.
  So we conclude that $\D H(\bar X, \bar P)$ is surjective (because $\D h(\bar
  X)$ is surjective) and $H$ is a local defining function for
  $\lift^{-1}(\sanifold)$.
  It follows that $\lift^{-1}(\sanifold)$ is an embedded submanifold of
  $\desing_s$ whose dimension is $\dim \desing_s - (m + n - s)s + d = \dim
  \sanifold + (n - r)(r - s)$ (Lemma~\ref{lemma:stratum-preimage} gives the
  dimension of $\desing_s$).
  The map $H$ is $\smooth{k}$ if $h$ is $\smooth{k}$, so we deduce that $\lift^{-1}(\sanifold)$ is $\smooth{k}$ if $\sanifold$ is $\smooth{k}$.

  We now derive the expression for the tangent space.
  Let $c \colon \reals \to \lift^{-1}(\sanifold)$ be a smooth curve such that
  $c(0) = (\bar X, \bar P)$.
  Then $t \mapsto \lift(c(t))$ is a smooth curve on $\sanifold$ passing through
  $\bar X$.
  In particular, we deduce that $\tangent_{(\bar X, \bar P)}
  \lift^{-1}(\sanifold) \subseteq \mathcal{F} = \{ (\dot X, \dot P) \in
  \tangent_{(\bar X, \bar P)} \desing : \dot X \in \tangent_{\bar X} \sanifold
  \}$.
  Let $(U, \Sigma, V)$ be a representation of $(\bar X, \bar P)$.
  Define the linear map $\ell \colon \mathcal{F} \to \tangent_{\bar X} \sanifold
  \times \reals^{(n - r) \times (r - s)}$ as $\ell(\dot X, \dot P) = (\dot
  X, V_\perp^\top \dot P V J)$, where $J \in \reals^{r \times (r - s)}$ is composed of the last $r - s$ columns of $I_r$.
  Let $(\dot X, \dot P)$ in the kernel of $\ell$ have representation $(K, V_\perp L)$.
  The expression of the tangent space in Proposition~\ref{prop:tangent-space} reveals that $K = \zeros$ and $L = \zeros$.
  So $\ell$ is injective, showing that $\dim \mathcal{F} \leq \dim \sanifold + (n - r)(r - s)$.
  This is the dimension of $\lift^{-1}(\sanifold)$ computed above so we conclude
  that $\tangent_{(\bar X, \bar P)} \lift^{-1}(\sanifold) = \mathcal{F}$.
\end{proof}

\begin{remark}\label{rmk:cst-rank}
  The conclusion of Proposition~\ref{prop:submanifold-preimage} may not hold if
  we remove the assumption that the matrices in $\sanifold$ all have the same
  rank.
  Take for example $m = n = 3$ and $r = 1$.
  Consider the set $\sanifold = \{\diag(t, 0, 0) : t \in \reals\}$, which is an
  embedded submanifold of $\reals^{m \times n}$.
  Then $\sanifold$ partitions into submanifolds $\sanifold_0 =
  \{\diag(0, 0, 0)\}$ and $\sanifold_{\neq 0} = \{\diag(t, 0, 0) : t \neq 0\}$,
  of dimensions $0$ and $1$ respectively.
  Accordingly, $\lift^{-1}(\sanifold) = \lift^{-1}(\sanifold_0) \cup
  \lift^{-1}(\sanifold_{\neq 0})$.
  Proposition~\ref{prop:submanifold-preimage} gives that these two preimages are
  embedded submanifolds of $\desing$ of dimensions $2$ and $1$ respectively, and they are disjoint.
  Their union is not an embedded submanifold.
\end{remark}

\begin{remark}\label{rmk:max-fiber-diffeomorphic}
  \TODOF{SIAM version does not have def of the preimage.}
  The maximal rank preimage $\lift^{-1}(\reals_r^{m \times n}) = \{(X, P) \in
  \desing : \rank X = r\}$ is an open submanifold of $\desing$ (of full
  dimension).
  This follows from Lemma~\ref{lemma:stratum-preimage}
  and~\citep[Prop.~5.1]{lee2013smooth}.
  In fact, $\lift^{-1}(\reals_r^{m \times n})$ is diffeomorphic to $\reals_r^{m
    \times n}$.
  Indeed, define the $\smooth{\infty}$ map $F \colon \fixedranksub \to
  \fixedrank$ as $F(X, P) = X$.
  It is bijective because when $\rank(X) = r$ there is a unique orthogonal
  projector $P$ of rank $n - r$ such that $XP = \zeros$.
  Let $X \in \fixedrank$ have thin SVD $X = U \Sigma V^\top$.
  Then the inverse is $F^{-1}(X) = (X, I - VV^\top)$.
  It can also be rewritten as $F^{-1}(X) = (X, I - X^\dagger X)$, revealing that
  $F^{-1}$ is smooth.
  (This is because the pseudo-inverse is smooth on matrices of fixed rank.)
\end{remark}

\section{Riemannian geometry}\label{sec:riemann}

We now endow the tangent spaces of $\desing$ with inner products to turn it into a Riemannian manifold.
This provides a notion of gradients, Hessians, and second-order retractions.

\subsection{\texorpdfstring{A family of metrics with parameter $\boldsymbol{\alpha}$}
                              {A family of metrics with parameter alpha}}\label{subsec:metric}

We equip the embedding space $\embeddingspace = \reals^{m \times n} \times
\sym(n)$ with the inner product
\begin{align}\label{eq:inner-product}
  \inner{\cdot}{\cdot}_\alpha \colon \embeddingspace \times \embeddingspace \to \reals && \text{defined as} && \sinner{(Y_1, Z_1)}{(Y_2, Z_2)}{\big}_\alpha = \inner{Y_1}{Y_2} + \alpha \inner{Z_1}{Z_2},
\end{align}
for some constant parameter $\alpha > 0$, where the inner products on the
right-hand side are the standard Frobenius ones.
This generalizes the canonical choice $\alpha = 1$.
We usually omit the subscript $\alpha$ because there is no ambiguity.
The inner product on $\embeddingspace$ induces inner products on the tangent spaces of
$\desing$, turning it into a Riemannian submanifold of $\embeddingspace$.
The set $\desing$ is closed in $\embeddingspace$ (Proposition~\ref{prop:desing-closed}) so we deduce
that $\desing$ is a complete manifold~\cite[Ex.~10.12]{boumal2020introduction}.

Such inner products can be computed in $O((m + n)r)$ operations
using the representation of tangent vectors in Definition~\ref{def:tgt-repr}.
For this, we define the diagonal matrix
\begin{align}\label{eq:sfactor}
  \sfactor(\alpha) = 2 \alpha I_r + \Sigma^2,
\end{align}
which appears in multiple expressions below.

\begin{proposition}\label{prop:inner-product}
  Given two tangent vectors $(\dot X_1, \dot P_1)$ and $(\dot X_2, \dot P_2)$ at
  $(X, P)$ with representations $(K_1, V_{p_1})$ and $(K_2, V_{p_2})$, their
  inner product is
  \begin{align*}
    \sinner{(\dot X_1, \dot P_1)}{(\dot X_2, \dot P_2)}{\big} = \inner{K_1}{K_2} + \inner{V_{p_1}}{V_{p_2} \sfactor(\alpha)},
  \end{align*}
  \TODOF{SIAM has slightly different phrasing here:}
  where $\sfactor(\alpha)$ is defined in~\eqref{eq:sfactor}.
  The induced norm is $\|(\dot X_1, \dot P_1)\| = \sqrt{\|K_1\|_\frob^2 +
    \|V_{p_1} \sfactor(\alpha)^{1/2}\|_\frob^2}$.
\end{proposition}
\begin{proof}
  From the representations (Definition~\ref{def:tgt-repr}), we find that the
  inner product~\eqref{eq:inner-product} is
  \begin{align*}
    \sinner{(\dot X_1, \dot P_1)}{(\dot X_2, \dot P_2)}{\big} &= \sinner{K_1^{}V^\top + U\Sigma V_{p_1}^\top}{K_2^{}V^\top + U\Sigma V_{p_2}^\top}{\big} \\ & \qquad + \alpha \sinner{V_{p_1}V^\top + VV_{p_1}^\top}{V_{p_2}V^\top + VV_{p_2}^\top}{\big}\\
                                                       &= \inner{K_1}{K_2} + \inner{V_{p_1}\Sigma}{V_{p_2}\Sigma} + 2\alpha\inner{V_{p_1}}{V_{p_2}},
  \end{align*}
  which gives the announced formula.
\end{proof}

At a point $(X, P) \in \desing$, the inner product~\eqref{eq:inner-product}
induces a normal space $\normal_{(X, P)}\desing$.
It is the orthogonal complement of $\tangent_{(X, P)}\desing$ in $\embeddingspace$.

\begin{proposition}\label{prop:normal-space}
  At $(X, P) \in \desing$ with representation $(U, \Sigma, V)$, the normal space is
  \begin{equation}\label{eq:normal-space}
    \begin{aligned}
      \normal_{(X, P)} \desing &= \Big\{ \big(U A V_\perp^\top + U_\perp^{} B V_\perp^\top, V C V^\top + V_\perp^{} D\transpose V^\top + V D V_\perp^\top + V_\perp^{} E V_\perp^\top\big)\\
                               &\qquad\quad: \Sigma A - 2 \alpha D = \zeros\Big\},
    \end{aligned}
  \end{equation}
  where $A, B, C, D, E$ are matrices with appropriate sizes and $C =
  C^\top$, $E = E^\top$.
\end{proposition}
\begin{proof}
  Let $(G, H)$ be in the normal space of $\desing$ at $(X, P)$.
  There exist matrices $A, B, C, D, E$ and $N_1, N_2$ of appropriate sizes such
  that $C = C^\top$, $E = E^\top$,
  \begin{align*}
    G & = UN_1V^\top + U_\perp^{} N_2V^\top + UAV_\perp^\top + U_\perp^{} BV_\perp^\top, \textrm{ and} \\
    H & = VCV^\top + V_\perp^{} D\transpose V^\top + V D V_\perp^\top + V_\perp^{} E V_\perp^\top.
  \end{align*}
  Let also $(\dot X, \dot P) = (KV^\top + U \Sigma L\transpose V_\perp^\top, -V_\perp
  L V^\top - V L\transpose V_\perp^\top)$ be a tangent vector parameterized by $K$
  and $L$ (Proposition~\ref{prop:tangent-space}).
  We compute the inner product
  \begin{align*}
    \sinner{(\dot X, \dot P)}{(G, H)}{\big} = \inner{K}{UN_1 + U_\perp N_2} + \inner{L^\top}{\Sigma A - 2 \alpha D}.
  \end{align*}
  This is zero for all matrices $K, L$ if and only if $N_1 = \zeros$,
  $N_2 = \zeros$ and $\Sigma A - 2 \alpha D = \zeros$.
\end{proof}

From the normal spaces we can deduce an expression for the orthogonal projectors
onto tangent spaces.
We use them in Section~\ref{subsec:derivatives} to compute Riemannian gradients
and Hessians.
The orthogonal projection of $(Y, Z) \in \embeddingspace$ onto $\tangent_{(X, P)}\desing$
is
\begin{align*}
  \proj_{(X, P)}(Y, Z) = \argmin_{(\dot X, \dot P) \in \tangent_{(X, P)} \desing} \big\|(\dot X - Y, \dot P - Z)\big\|^2,
\end{align*}
where the norm is induced by the inner product~\eqref{eq:inner-product}.
The (unique) solution $(\dot X, \dot P)$ is a tangent vector at $(X, P)$.
We seek a representation (Definition~\ref{def:tgt-repr}) for it.
\TODOF{SIAM version does not have Lemma~\ref{lemma:proj-x-p}.}
(See also Lemma~\ref{lemma:proj-x-p}.)

\begin{proposition}\label{prop:proj}
  Given $(X, P) \in \desing$ with representation $(U, \Sigma, V)$ and $(Y, Z)
  \in \embeddingspace$, the projection $\proj_{(X, P)}(Y, Z)$ is a tangent vector with
  representation
  \begin{align*}
    K = YV && \text{and} && V_p = P\pig(Y\transpose U \Sigma - 2 \alpha Z V\pig)\sfactor(\alpha)^{-1}.
  \end{align*}
\end{proposition}
\begin{proof}
  The projection of $(Y, Z)$ on the tangent space at $(X, P)$ can be written as
  \begin{align*}
    \proj_{(X, P)}(Y, Z) &= \big(KV^\top + U \Sigma L\transpose V_\perp^\top, -V_\perp^{} L V^\top - V L\transpose V_\perp^\top\big)
  \end{align*}
  for some matrices $K$ and $L$ (Proposition~\ref{prop:tangent-space}).
  Moreover, the vector $(Y, Z) - \proj_{(X, P)}(Y, Z)$ is in the normal space
  $\normal_{(X, P)} \desing$.
  Hence, Proposition~\ref{prop:normal-space} gives that
  \begin{align*}
    \begin{cases}
      Y - KV^\top - U \Sigma L\transpose V_\perp^\top = UAV_\perp^\top + U_\perp^{} BV_\perp^\top, \\
      Z + V_\perp^{} L V^\top + V L\transpose V_\perp^\top = VCV^\top + V_\perp^{} D\transpose V^\top + V D V_\perp^\top + V_\perp^{} E V_\perp^\top
    \end{cases}
  \end{align*}
  for some matrices $A, B, C, D, E$ of appropriate sizes satisfying $\Sigma A -
  2 \alpha D = \zeros$, $C = C^\top$ and $E = E^\top$.
  The first equality gives $K = Y V$.
  We also get $U\transpose Y V_\perp - \Sigma L^\top = A$.
  The second equality gives $V\transpose Z V_\perp + L^\top = D$.
  Combining these last two relations and $Z = Z^\top$ yields
  \begin{align*}
    \zeros = A\transpose \Sigma - 2 \alpha D^\top = V_\perp\transpose \pig( Y\transpose U \Sigma - 2 \alpha Z V \pig) - L\sfactor(\alpha),
  \end{align*}
  from which we deduce the expressions for $L$ and $V_p = V_\perp L$.
\end{proof}

\begin{remark}\label{rmk:max-fiber-isometric}
  When $\alpha = 0$,
  the bilinear form~\eqref{eq:inner-product} is still positive definite on
  the tangent spaces of $\fixedranksub$.
  Indeed, in this case we have $\inner{(\dot X, \dot P)}{(\dot X, \dot P)} =
  \|\dot X\|_\frob^2$ for a given tangent vector $(\dot X, \dot P)$ at $(X,
  P)$.
  This is zero exactly if $\dot X = 0$;
  yet when $\rank X = r$ and $\dot X = 0$
  we also have $\dot P = 0$ by~\eqref{eq:tangent-space}.
  This proves that the inner product is still positive definite when $\alpha = 0$.
  Recall from Remark~\ref{rmk:max-fiber-diffeomorphic} that the manifolds
  $\fixedranksub$ and $\reals_r^{m \times n}$ are diffeomorphic.
  When $\alpha = 0$ they are additionally isometric.
  (It is easy to see that the differential of the map $F$ from
  Remark~\ref{rmk:max-fiber-diffeomorphic} preserves the inner products.)
  We deduce that algorithms running on $\reals_r^{m \times n}$ can equivalently
  be seen as running on $\fixedranksub$ with $\alpha = 0$.
\end{remark}

\subsection{Gradients and Hessians}\label{subsec:derivatives}

Consider a function $\sfc \colon \reals^{m \times n} \to \reals$ and let $\mfc =
\sfc \circ \lift$ be the lifted function over $\desing$.
Let $\grad \sfc$ and $\hess \sfc$ denote the Euclidean gradient and Hessian of
$\sfc$, and $\grad \mfc$ and $\hess \mfc$ denote the Riemannian gradient and
Hessian of $\mfc$.
In this section we express $\grad \mfc$ and $\hess \mfc$ in terms of $\grad
\sfc$ and $\hess \sfc$.
Importantly, $\mfc$ depends only on the first component of the pair $(X, P)$.
This considerably simplifies expressions for the derivatives.

\paragraph{First-order derivatives.}\label{subsec:1derivatives}

We first find a representation (Definition~\ref{def:tgt-repr}) for $\grad \mfc$.

\begin{proposition}\label{prop:gradient}
  The Riemannian gradient $\grad \mfc(X, P)$ is a tangent vector with representation
  \begin{align*}
      K = \grad \sfc(X) V && \text{and} && V_p = P \grad \sfc(X)\transpose U \Sigma \sfactor(\alpha)^{-1},
  \end{align*}
  where $(U, \Sigma, V)$ is a representation of $(X, P)$ and $\sfactor(\alpha)$
  is as in~\eqref{eq:sfactor}.
\end{proposition}
\begin{proof}
  Formally, $\mfc(X, P) = \sfc(X)$ is defined on $\desing$.
  Let $\bar \mfc(X, P) = \sfc(X)$ denote an extension of $\mfc$ to $\embeddingspace$.
  Its Euclidean gradient is $\grad \bar \mfc(X, P) = (\grad \sfc(X), 0)$.
  As $\desing$ is a Riemannian submanifold of $\embeddingspace$, the Riemannian gradient
  is the projection of the Euclidean gradient onto the current tangent
  space~\cite[Prop.~3.61]{boumal2020introduction}.
  We deduce that $\grad \mfc(X, P) = \proj_{(X, P)}(\grad \sfc(X), \zeros)$ and
  conclude with Proposition~\ref{prop:proj}.
\end{proof}

The Riemannian gradient can be computed efficiently if $\grad \sfc(X)\transpose U$ and $\grad \sfc(X)V$ can be computed efficiently---see Table~\ref{tab:complexities}.
This can be the case when $\grad \sfc(X)$ is appropriately structured (for example if it is sparse).
From the expression of the gradient we deduce the first-order optimality conditions.

\begin{corollary}\label{cor:critical-point}
  A point $(X, P) \in \desing$ with representation $(U, \Sigma, V)$ is critical
  for $\mfc$ if and only if $\grad \sfc(X) V = \zeros$ and $\grad \sfc(X)\transpose U
  \Sigma = \zeros$.
\end{corollary}
\begin{proof}
  Let $(K, V_p)$ be a representation of $\grad \mfc(X, P)$ (Proposition~\ref{prop:gradient}).
  The stated conditions are equivalent to $K = \zeros$ and $V_p = \zeros$.
\end{proof}

\paragraph{Second-order derivatives.}\label{subsec:2derivatives}

In this section we assume that $\sfc$ (and hence $\mfc$) is $\smooth{2}$.
We derive the Hessian of $\mfc$ as a function of the Hessian of $\sfc$.
The proof is computational and can be found in Appendix~\ref{sec:appendix-second-order-derivatives}.
\TODOF{SIAM version: we omit the proof.}
Notice again that $\hess \mfc$ can be computed efficiently if $\grad \sfc$ and $\hess \sfc$ are appropriately structured---see Table~\ref{tab:complexities}.

\begin{proposition}\label{prop:hessian}
  Let $(X, P) \in \desing$ and $(\dot X, \dot P) \in \tangent_{(X, P)}\desing$ have representations $(U, \Sigma, V)$ and $(K, V_p)$.
  Then $\hess \mfc(X, P)[\dot X, \dot P]$ is a tangent vector at $(X, P)$ with
  representation $(\bar K, \bar V_p)$ satisfying
  \begin{align*}
    \begin{cases}
      \bar K = \hess \sfc(X)[\dot X] V + M \grad\sfc(X)V_p\\
      \bar V_p = P \Big(\hess \sfc(X)[\dot X]\transpose U \Sigma + \grad\sfc(X)\transpose M K\Big)\sfactor(\alpha)^{-1}
    \end{cases}
    && \text{where} && M = I - U\Sigma^2\sfactor(\alpha)^{-1} U^\top.
  \end{align*}
\end{proposition}

From this, we deduce the second-order optimality conditions.
Notice that these are independent of $\alpha$, which is expected since first- and second-order criticality conditions do not depend on the Riemannian metric.

\begin{corollary}\label{cor:hess-inner-product}
  Let $(X, P) \in \desing$ and $(\dot X, \dot P) \in \tangent_{(X, P)} \desing$ have representations $(U, \Sigma, V)$ and $(K, V_p)$.
  Define $M = I - U\Sigma^2\sfactor(\alpha)^{-1} U^\top$.
  Then
  \begin{align*}
    \inner{(\dot X, \dot P)}{\hess \mfc(X, P)[\dot X, \dot P]} = \inner{\dot X}{\hess \sfc(X)[\dot X]} + 2 \inner{K}{M \grad \sfc(X) V_p}.
  \end{align*}
  In particular, the second-order optimality conditions are
  \begin{align*}
    \grad \sfc(X) V = \zeros, && \grad \sfc(X)\transpose U \Sigma = \zeros && \text{and} && \inner{\dot X}{\hess \sfc(X)[\dot X]} + 2 \inner{K}{\grad \sfc(X) V_p} \geq 0
  \end{align*}
  for all $(\dot X, \dot P) \in \tangent_{(X, P)} \desing$ (with representation $(K, V_p)$).
\end{corollary}
\begin{proof}
  Let $\gamma = \inner{(\dot X, \dot P)}{\hess \mfc(X, P)[\dot X, \dot P]}$
  denote the Hessian in quadratic form.
  Propositions~\ref{prop:inner-product} and~\ref{prop:hessian} give
  \begin{align*}
    \gamma &= \inner{K}{\hess \sfc(X)[\dot X] V + M \grad \sfc(X) V_p} + \inner{V_p}{\hess \sfc(X)[\dot X]\transpose U \Sigma + \grad\sfc(X)\transpose M K}\\
           &= \inner{K V\transpose + U \Sigma V_p^\top}{\hess \sfc(X)[\dot X]} + 2 \inner{K}{M \grad \sfc(X) V_p},
  \end{align*}
  where we used the equality $P V_p = V_p$ and grouped the terms.
  Notice that $K V^\top + U \Sigma V_p^\top = \dot X$
  (Definition~\ref{def:tgt-repr}) so we obtain the expression for $\gamma$.
  The first-order optimality conditions (Corollary~\ref{cor:critical-point})
  further imply $M \grad \sfc(X) = \grad \sfc(X)$.
\end{proof}

\subsection{Second-order retractions}\label{subsec:2-retractions}

Second-order retractions are retractions (Definition~\ref{def:retr}) that
satisfy an additional regularity
requirement~\citep[Prop.~5.5.5]{absil2008optimization}.
As they provide better approximations to geodesics, they can be useful both for numerical and theoretical reasons.
In this section we list a couple that are efficiently computable.

To define second-order retractions, we need the following notion.
A curve $c \colon \reals \to \desing$ is also a curve in $\embeddingspace$.
Let $\acc{c}$ denote its (extrinsic) acceleration in $\embeddingspace$.
Further let $\iacc{c}$ denote its (intrinsic, Riemannian) acceleration as a curve on $\desing$.
It holds that $\iacc{c}(t) = \proj_{\tangent_{c(t)} \desing} \acc{c}(t)$.

\begin{definition}\label{def:second-order-retr}
  A retraction $\retr \colon \tangent \desing \to \desing$ is
  \emph{second-order} if for all $(X, P) \in \desing$ and $(\dot X, \dot P) \in
  \tangent_{(X, P)}\desing$ the curve $c(t) = \retr_{(X, P)}(t \dot X, t \dot P)$
  satisfies $\iacc{c}(0) = \zeros$.
\end{definition}

\paragraph{A second-order retraction: metric projection.} For Riemannian
submanifolds of a Euclidean space, a natural way to define the retraction
consists in taking a step in the embedding space and to project back onto the
manifold.
This is the \emph{metric projection retraction}, and it is
second order~\cite[Ex.~23]{absil2012projection}, \citep[\S5.12]{boumal2020introduction}.
More precisely, given a point $(X, P) \in \desing$ and a tangent vector
$(\dot X, \dot P) \in \tangent_{(X, P)}\desing$, metric projection
is
\begin{align}\label{eq:proj-retr}
  \retr_{(X, P)}(\dot X, \dot P) &\in \argmin_{(\bar X, \bar P) \in \desing} \big\|\big(X + \dot X - \bar X, P + \dot P - \bar P\big)\big\|^2,
\end{align}
where the norm is induced by~\eqref{eq:inner-product}.
(The solution is unique for small tangent vectors but not in general: more on this below.)
Algorithm~\ref{alg:metric-retr} describes an efficient procedure with complexity
$O((m + n)r^2)$ to compute a representation $(\bar U, \bar \Sigma, \bar V)$
for a solution to~\eqref{eq:proj-retr}.
We justify this with a lemma then a proposition.

\newcommand{\ccc}{h}
\begin{lemma}\label{lemma:opt-retr-handy}
  Given $A \in \reals^{m \times n}$ and $B \in \sym(n)$, consider the cost function
  \begin{align}\label{eq:opt-retr-handy}
    \ccc(U, \Sigma, V) = \|U \Sigma V^\top - A\|_\frob^2 + \alpha\|VV^\top - B\|_\frob^2
  \end{align}
  on the set $\stiefel(m, r) \times \reals^{r \times r} \times \stiefel(n, r)$.
  Define $C = \frac{1}{2}A\transpose A + \alpha B \in \sym(n)$ and let $V \in \stiefel(n, r)$ contain $r$ eigenvectors of $C$ associated with its $r$ largest eigenvalues.
  Let $A V = U \Sigma H^\top$ be a thin SVD.
  Then $(U, \Sigma, V H)$ is a minimizer of $\ccc$.
\end{lemma}
\begin{proof}
  With $C = \frac{1}{2}A\transpose A + \alpha B$, we can rewrite $h$ as
  \begin{align*}
    \ccc(U, \Sigma, V) &= \|U\Sigma - AV\|_\frob^2 + \|A(I - VV^\top)\|_\frob^2 + \alpha\|VV^\top - B\|_\frob^2\\
                       &= \|U \Sigma - AV\|_\frob^2 + \|VV^\top - C\|_\frob^2 + \|A\|_\frob^2 + \alpha\|B\|_\frob^2 - \|C\|_\frob^2 + (\alpha - 1)r.
  \end{align*}
  The second equality is because $\|A(I - VV^\top)\|_\frob^2 = \|A\|_\frob^2 -
  \inner{A\transpose A}{VV^\top}$, $\|VV^\top - B\|_\frob^2 = \|B\|_\frob^2 - 2
  \inner{B}{VV^\top} + r$ and $\|VV^\top - C\|_\frob^2 = \|C\|_\frob^2 - 2
  \inner{C}{VV^\top} + r$.
  Observe that the first term is nonnegative in the expression of $h$ above.
  It follows that
  \begin{equation*}
    \min_{U, \Sigma, V} h(U, \Sigma, V) \geq \min_{V \in \stiefel(n, r)} \|VV^\top - C\|_\frob^2 + \|A\|_\frob^2 + \alpha\|B\|_\frob^2 - \|C\|_\frob^2 + (\alpha - 1)r.
  \end{equation*}
  Let $V$, $U$, $\Sigma$ and $H$ be as prescribed in the lemma statement.
  Then the cost $h(U, \Sigma, VH)$ is equal to the right-hand side above so $(U, \Sigma, VH)$ minimizes $h$.
\end{proof}

\begin{proposition}\label{prop:metric-retr}
  Let $(X, P) \in \desing$ and $(\dot X, \dot P) \in \tangent_{(X, P)}\desing$
  have representations $(U, \Sigma, V)$ and $(K, V_p)$.
  Define $D \in \sym(2r)$ as
  \begin{align}\label{eq:retr-D}
      D =
      \begin{bmatrix}
        \Sigma^2 + \Sigma U\transpose K + K\transpose U \Sigma + K\transpose K + 2\alpha I & \Sigma^2 + K\transpose U \Sigma + 2\alpha I\\
        \Sigma^2 + \Sigma U\transpose K + 2\alpha I & \Sigma^2
      \end{bmatrix}
      .
  \end{align}
  Then Algorithm~\ref{alg:metric-retr} computes a metric projection retraction~\eqref{eq:proj-retr} at $(X, P)$ in direction $(\dot X, \dot P)$.
\end{proposition}
\begin{proof}
  Apply Lemma~\ref{lemma:opt-retr-handy} to $A = X + \dot X$ and $B = I - P - \dot P$.
  The matrix $C = \frac{1}{2}A\transpose A + \alpha B$ satisfies
  \begin{align*}
    2C = (X + \dot X)\transpose (X + \dot X) + 2\alpha(I - P - \dot P) = \begin{bmatrix} V & V_p \end{bmatrix} D \begin{bmatrix} V & V_p \end{bmatrix}^\top,
  \end{align*}
  where $D$ is defined in~\eqref{eq:retr-D}.
  It shows that Algorithm~\ref{alg:metric-retr} computes a minimizer of~\eqref{eq:opt-retr-handy}, and it is a metric projection by definition~\eqref{eq:proj-retr}.
\end{proof}

\TODOF{Check article of \cite{vandereycken2013low} to see if it's better (for
  numerical stability) to do QR on the concatenation $(V V_p)$ rather than on
  $V_p$ alone.
  It seems like Bart doesn't do that (see Algorithm 6 in his article).}

\begin{algorithm}[t]\caption{Metric projection retraction}\label{alg:metric-retr}
  \begin{algorithmic}[1]
    \State \textbf{Input:} A point $(X, P)$ in $\desing$ with representation
    $(U, \Sigma, V)$ and a tangent vector $(\dot X, \dot P)$ in $\tangent_{(X,
      P)}\desing$ with representation $(K, V_p)$.
    \State Compute a thin $\mathrm{QR}$ decomposition $\begin{bmatrix} V &
      V_p \end{bmatrix} = QR$.
    \TODOF{We could compute a QR of $V_p$ only but the resulting algorithm
      description is not as compact.}
    \State Construct $D$ as in~\eqref{eq:retr-D}.
    \State Find the top $r$ eigenvectors $\tilde{U} \in \stiefel(2r, r)$ of the
    matrix $RDR^\top$.
    \State Compute $\tilde V = Q\tilde{U}$.
    \State Compute a thin SVD $\bar U \bar \Sigma H^\top$ of $(X + \dot X) \tilde V = (U
    \Sigma  + K)V\transpose \tilde V + U \Sigma V_p\transpose \tilde V$.
    \State \textbf{Output:} The triplet $(\bar U, \bar \Sigma, \bar V)$ where $\bar V = \tilde V H$.
  \end{algorithmic}
\end{algorithm}

  The metric projection retraction is not uniquely defined everywhere (and hence
  is not smooth everywhere).
  To see this, consider a point with representation $(U, \Sigma, V)$ and a
  tangent vector with representation $(K, V_\perp L)$.
  Choose $\Sigma = I_r$, $K = -(2 \alpha + 1) U$ and $L = (4 \alpha^2 + 2 \alpha)^{\sfrac{1}{2}} J$, where $J \in \reals^{(n - r) \times r}$ contains the first $r$ columns of $I_{n - r}$ (assuming $r \leq n - r$).
  Then the matrix $C$ defined in the proof of Proposition~\ref{prop:metric-retr} is
  \begin{align*}
    C = 2\alpha(2 \alpha + 1) \begin{bmatrix} V & V_\perp \end{bmatrix}
    \begin{bmatrix} I_{2r} & \zeros\\ \zeros & \zeros \end{bmatrix}
    \begin{bmatrix} V & V_\perp \end{bmatrix}^\top.
  \end{align*}
  The top $2r$ eigenvalues of $C$ are the same, so the metric projection is not unique (see Lemma~\ref{lemma:opt-retr-handy}).
  This motivates the next part.

\paragraph{A globally smooth second-order retraction.}

We proceed to derive a computable second-order retraction that is smoothly
defined on the whole tangent bundle.
It is based on the polar retraction for the Grassmann manifold~\citep[Example~9.34]{boumal2020introduction} (as opposed to the Q-factor example in Section~\ref{subsec:retractions}).
\TODOF{SIAM version: no proof. Refer to Arxiv.}
The proof is provided in Appendix~\ref{sec:second-order-retractions}.

\begin{proposition}\label{prop:global-retr}
  Let $(X, P) \in \desing$ and $(\dot X, \dot P) \in \tangent_{(X, P)} \desing$
  be represented by $(U, \Sigma, V)$ and $(K, V_p)$.
  Define $Z = V + V_p\big(I_r - K\transpose U \Sigma \sfactor(\alpha)^{-1}\big)$ and
  the map
  \begin{align*}
    \retr_{(X, P)}(\dot X, \dot P) = \big( (X + \dot X)Q Q^\top, I - Q Q^\top \big) && \text{where} && Q = Z (Z\transpose Z)^{-1/2}.
  \end{align*}
  Then $\retr$ is a second-order retraction.
\end{proposition}

In Proposition~\ref{prop:global-retr}, the matrix $Z \in \reals^{n \times r}$ always has rank $r$, so the retraction is well-defined on all of $\tangent \desing$ (unlike the metric projection retraction).
It is also possible to compute a representation of the retracted point in $O((m
+ n)r^2)$ operations.
Indeed, building the matrices $Z$ and $Q$ takes $O((m + n)r^2)$ operations.
Then building the matrix $W = (X + \dot X) Q = (U \Sigma + K) V\transpose Q + (U
\Sigma) V_p\transpose Q$ takes $O((m + n)r^2)$ operations.
Now let $W = \bar U \bar \Sigma H^\top \in \reals^{m \times r}$ be a thin SVD,
which can be computed in $O(mr^2)$ operations.
Then $(\bar X, \bar P) = \retr_{(X, P)}(\dot X, \dot P)$ has representation
$(\bar U, \bar \Sigma, \bar V)$, where $\bar V = Q H$.

\section{Convergence guarantees}\label{sec:conv-guarantees}

In this section we use the desingularization to obtain convergence
guarantees for low-rank optimization.
Again, $\sfc \colon \reals^{m \times n} \to \reals$ is the function we want to
minimize over $\boundedrank$---that is problem~\eqref{eq:problem}---and
$\mfc = \sfc \circ \lift$ is the lifted function over $\desing$.
The reparameterized optimization problem is hence
\begin{align}\label{eq:lifted-problem}\tag{Q}
  \min_{(X, P) \in \desing}\, \mfc(X, P) && \textrm{ with } && \mfc(X, P) = \sfc(X).
\end{align}

We first establish a global convergence theorem.
Specifically, we derive algorithms whose iterates accumulate only at stationary points of~\eqref{eq:problem}.
This is not straightforward because the feasible set $\boundedrank$ is nonsmooth.
Consequently, a sequence of points in $\boundedrank$ may converge to a non-stationary point, even when the stationarity measure of the sequence converges to zero.
\cite{levin2023finding} exhibited such sequences generated by the algorithm P$^2$GD~\citep{schneider2015convergence}.
They also designed second-order algorithms that provably accumulate only at stationary points of~\eqref{eq:problem} using reparameterizations through various $\varphi$ (including the desingularization).
\cite{themelis2018forward,olikier2023apocalypse,pauwels2024generic,olikier2025low}, and~\cite{olikier2025projected} obtained guarantees for first-order algorithms, including projected gradient descent.

We propose returning to the reparameterization approach, here with the desingularization, because it allows for second-order algorithms.
As we detail later, these methods further enjoy fast local convergence.

The global convergence theorem below is a variant of~\cite[Thm.~1.1 \& Ex.~3.13]{levin2023finding}.
The argument is the same, but here it is applied to any \emph{descent} algorithm, that is, one that generates sequences of iterates with non-increasing function values.

\begin{theorem}[Global convergence]\label{th:global}
  For $(X_0, P_0) \in \desing$, suppose $\{X \in \boundedrank : \sfc(X)
  \leq \sfc(X_0)\}$ is compact.
  Then $\{(X, P) \in \desing : \mfc(X, P) \leq \mfc(X_0, P_0)\}$ is also
  compact.
  In particular, \emph{descent} algorithms initialized at $(X_0, P_0)$
  generate sequences $\sequence{(X_k, P_k)}$ that accumulate to at least one point.
  Furthermore, if $\sfc$ is $\smooth{2}$ and all accumulation points of
  $\sequence{(X_k, P_k)}$ are second-order critical
  for~\eqref{eq:lifted-problem} then all accumulation points of $\sequence{X_k}$
  are stationary for~\eqref{eq:problem}.
\end{theorem}

Many standard algorithms satisfy the descent condition, such as gradient descent
(with line search or appropriate constant step size) and trust-region methods.
Except in some pathological cases, they also accumulate only at
\emph{second}-order critical points.

We additionally aim for fast local convergence.
In general, the minimizers of~\eqref{eq:lifted-problem} are not isolated.
This is because the fiber $\lift^{-1}(X)$ of a local minimizer $X$ of~\eqref{eq:problem} is a continuum of local minimizers of~\eqref{eq:lifted-problem} whenever $\rank X < r$ (see Proposition~\ref{prop:submanifold-preimage}).
Consequently, we cannot assume local strong convexity.
In the literature, it is typical to resort instead to the \polyakloja{} (\pl{}) condition (see Definition~\ref{def:pl} below),
which requires the squared gradient norm to locally dominate the relative function value around local minimizers.
In fact, \pl{} is necessary for fast local convergence of gradient descent~\citep[Thm.~5]{abbaszadehpeivasti2023conditions}.

Accordingly, we want to obtain \pl{} for $\mfc$ based on hypotheses on $\sfc$.
This is nontrivial because the set $\boundedrank$ is nonsmooth.
The \pl{} condition is not even well defined for~\eqref{eq:problem} and we must assume something different.

For this reason, we rely on the \morsebott{} (\mb{}) property.
It is locally equivalent to \pl{} on manifolds and requires positive definiteness of the Hessian along the normal space to the set of minimizers.
In Section~\ref{subsec:local-conv-pl} we define the variant~\eqref{eq:mb-type},
a type of \mb{} condition for $\sfc$ that takes into account the nonsmoothness of $\boundedrank$, allowing us to deduce \pl{} for $\mfc$.

Our results encompass existing local convergence theory for minimizers of maximal rank.
They also enable fast convergence to minimizers of non-maximal rank.
As an illustration, if $\hess \sfc(X) \succ \zeros$ at a minimizer $X \in \boundedrank$ of~\eqref{eq:problem} then $\sfc$ satisfies~\eqref{eq:mb-type} at $X$.
In this simple case, the minimizer is (locally) unique, but the definition of~\eqref{eq:mb-type} allows for more complicated sets of minimizers.
See also Remark~\ref{rmk:feasibility-of-mb-type}.

\begin{theorem}[Local convergence]\label{th:local}
  Let $X$ be a local minimizer of~\eqref{eq:problem}.
  Suppose $\sfc$ is $\smooth{2}$ and satisfies the \morsebott{} type
  property~\eqref{eq:mb-type} at $X$.
  Then $\mfc$ satisfies \pl{} around $\lift^{-1}(X)$.
\end{theorem}

Many standard algorithms have fast local rates when \pl{} holds.
This includes gradient descent (linear), regularized Newton (quadratic) and
trust-regions (superlinear).
See~\citep{polyak1963gradient,nesterov2006cubic} for gradient descent and
regularized Newton in the Euclidean case,
and~\citep{rebjock2024fast,rebjock2024tcg} for trust-regions and extensions
to Riemannian manifolds as is needed here.

\subsection{Points of interest of the two problems}\label{subsec:points-of-interest}

To study convergence guarantees we must understand how the points of interest
of~\eqref{eq:problem} and~\eqref{eq:lifted-problem} relate.
The parameterization $\lift$ may indeed induce spurious critical points, that
is, points $(X, P) \in \desing$ that are critical for~\eqref{eq:lifted-problem}
but such that $\lift(X, P) = X$ is not stationary for~\eqref{eq:problem}.
However, there exist some (already known) relations that we briefly mention
here.

If $X \in \boundedrank$ is a stationary point (respectively a local
minimizer) for~\eqref{eq:problem} then the fiber $\lift^{-1}(X)$ is composed of
critical points (respectively local minimizers) for~\eqref{eq:lifted-problem}: this
is so because $\lift$ is smooth.

Conversely, a critical point (respectively a local minimizer)
for~\eqref{eq:lifted-problem} may not be a stationary point (respectively a local minimizer)
for~\eqref{eq:problem}.
(This is partly because $\lift$ is not
open~\citep[Prop.~2.37]{levin2020towards}.)
However, the following known facts provide welcome guarantees:
\begin{itemize}
\item Let $(X, P) \in \desing$ with $\rank X = r$.
  If $(X, P)$ is a critical point (respectively a second-order critical point, a local minimizer) for~\eqref{eq:lifted-problem} then $X$ is a stationary point (respectively a second-order stationary point, a local minimizer) for~\eqref{eq:problem}.
\item If $(X, P)$ is second-order critical for~\eqref{eq:lifted-problem} then $X$ is stationary for~\eqref{eq:problem}.
\end{itemize}
See~\cite[\S2.3]{levin2020towards} and~\cite[Prop.~2.9]{levin2025effect} for proofs.
These facts are analogous to the ones established by~\citet{ha2020equivalence} for the LR parameterization.
In some cases, these properties can be used to deduce the strict saddle property for the LR parameterization, and therefore we expect the same should be true for the desingularization.

\subsection{Global convergence}\label{subsec:global-conv}

This section provides the ingredients at the basis of Theorem~\ref{th:global}.

\paragraph{Sublevel sets.}

A notable property of the parameterization $\lift$ is that it preserves
compactness of sublevel sets.
This is not the case for the parameterization $(L, R) \mapsto LR^\top$ because its fibers are
unbounded.

\begin{proposition}
  If~\eqref{eq:problem} has compact sublevel sets then~\eqref{eq:lifted-problem}
  also has compact sublevel sets.
\end{proposition}
\begin{proof}
  For all $\rho$, $\mfc^{-1}\big(\interval[open left]{-\infty}{\rho}\big) =
  \lift^{-1}\big(\sfc^{-1}\big(\interval[open left]{-\infty}{\rho}\big)\big)$ and
  $\lift$ is proper (Proposition~\ref{prop:lift-proper}).
\end{proof}

We deduce Theorem~\ref{th:global} from this and standard compactness arguments.
Indeed, iterates of descent algorithms stay in the initial sublevel
set.
If it is compact, they also accumulate to at least one point.
Section~\ref{subsec:points-of-interest} further provides the stationarity
guarantees.
This type of argument appears in~\cite[Thm.~3.5]{levin2023finding} for
trust-region methods in the context of low-rank optimization.
We cannot say the same for the LR parameterization because it is \emph{a priori}
not possible to ensure that the iterates remain in a bounded domain
(unless we modify the algorithm, e.g., with rebalancing).

\paragraph{Lipschitz continuity.}

Lipschitz properties of the cost function are often instrumental to obtain \emph{descent} algorithms and to ensure that iterates accumulate at first- or second-order critical points.
When $\mfc$ is $\smooth{2}$, standard continuity arguments readily give that $\grad \mfc$ is Lipschitz continuous in any compact domain.
The same is true for the Hessian $\hess \mfc$ when $\mfc$ is $\smooth{3}$.
It is possible to bound the Lipschitz constants and gain some control.
We now study $\grad \mfc$ and $\hess \mfc$ for this purpose.
\TODOF{We removed the proofs from the SIAM version.}

\begin{proposition}
  The bound $\|\grad \mfc(X, P)\| \leq \|\grad \sfc(X)\|_\frob$ holds for all
  $(X, P) \in \desing$.
\end{proposition}
\begin{proof}
  This follows from $\grad \mfc(X, P) = \proj_{(X, P)}(\grad \sfc(X), \zeros)$:
  see proof of Proposition~\ref{prop:gradient}.
\end{proof}

In particular, if $\sfc$ is $L$-Lipschitz continuous on $\reals^{m \times n}$ then $\mfc$ is $L$-Lipschitz continuous on $\desing$.
We further obtain a bound for the Hessian.

\begin{proposition}
  Assume $\sfc$ is $\smooth{2}$ on a neighborhood of $\boundedrank$.
  For all $(X, P) \in \desing$ the Hessian is bounded as
  \begin{align*}
    \|\hess \mfc(X, P)\|_\opnorm &\leq \|\hess \sfc(X)\|_\opnorm + \big(2\alpha + \sigma_r(X)^2\big)^{-\frac{1}{2}}\|\grad \sfc(X) P\|_2.
  \end{align*}
\end{proposition}
\begin{proof}
  Pick $(X, P) \in \desing$ and $(\dot X, \dot P) \in \tangent_{(X, P)} \desing$
  with representations $(U, \Sigma, V)$ and $(K, V_p)$.
  Proposition~\ref{prop:hessian} gives that $\hess \mfc(X, P)[\dot X, \dot P]$
  has representation $(\bar K, \bar V_p)$ where
  \begin{align*}
      \bar K = \hess\sfc(X)[\dot X] V + M \grad\sfc(X)V_p && \text{and} && \bar V_p = P\pig(\hess\sfc(X)[\dot X]\transpose U \Sigma + \grad\sfc(X)\transpose M K\pig) \sfactor(\alpha)^{-1},
  \end{align*}
  with $M = I - U\Sigma^2\sfactor(\alpha)^{-1} U^\top$.
  We aim to bound $\|\hess \mfc(X, P)[\dot X, \dot P]\|_\frob^2 =
  \|\bar K\|_\frob^2 + \|\bar V_p \sfactor(\alpha)^{1/2}\|_\frob^2$.
  First note that $\|\bar K\|_\frob^2 = a_1 + a_2 + a_3$ where
  \begin{align*}
    a_1 = \|\hess \sfc(X)[\dot X] V\|_\frob^2, && a_2 = 2\inner{\hess \sfc(X)[\dot X] V}{M \grad \sfc(X) V_p}, && \text{and} && a_3 = \|M \grad \sfc(X) V_p\|_\frob^2.
  \end{align*}
  Moreover we have $\|\bar V_p \sfactor(\alpha)^{1/2}\|_\frob^2 = b_1 + b_2 +
  b_3$ where
  \begin{align*}
    b_1 = \|P \hess \sfc(X)[\dot X]\transpose U \Sigma \sfactor(\alpha)^{-1/2}\|_\frob^2, && b_3 = \|P \grad\sfc(X)\transpose M K \sfactor(\alpha)^{-1/2}\|_\frob^2,
  \end{align*}
  and $b_2 = 2\inner{\hess \sfc(X)[\dot X]\transpose U \Sigma}{P \grad \sfc(X)^\top
    M K \sfactor(\alpha)^{-1}}$.
  We compute that
  \begin{align*}
    a_1 + b_1 &= \|\hess \sfc(X)[\dot X] (I - P)\|_\frob^2 + \|P \hess \sfc(X)[\dot X]\transpose U \Sigma \sfactor(\alpha)^{-1/2} U^\top\|_\frob^2\\
              &\leq \|\hess \sfc(X)[\dot X] (I - P)\|_\frob^2 + \|P \hess \sfc(X)[\dot X]^\top\|_\frob^2\\
              &= \|\hess \sfc(X)[\dot X]\|_\frob^2  \leq \|\hess \sfc(X)\|_\opnorm^2 \|\dot X\|_\frob^2,
  \end{align*}
  where the first inequality holds because $\|U \Sigma \sfactor(\alpha)^{-1/2}
  U^\top\|_2 \leq 1$ (recall the definition of $\sfactor$~\eqref{eq:sfactor})
  and using the bound $\|AB\|_\frob \leq \|A\|_2 \|B\|_\frob$ for matrix norms.
  Additionally,
  repeatedly using the latter bound
  and the fact that $\|M\|_2 \leq 1$,
  we find
  \begin{align*}
    a_3 + b_3 &= \|M \grad \sfc(X) P V_p\|_\frob^2 + \|P \grad \sfc(X)\transpose M K\sfactor(\alpha)^{-1/2}\|_\frob^2\\
              &\leq \big( \|V_p\|_\frob^2 + \|K \sfactor(\alpha)^{-1/2}\|_\frob^2 \big) \|\grad \sfc(X) P\|_2^2\\
              &\leq \big( \|V_p\|_\frob^2 + \big(2 \alpha + \sigma_r(X)^2\big)^{-1}\|K\|_\frob^2 \big) \|\grad \sfc(X) P\|_2^2\\
              &\leq \big(2 \alpha + \sigma_r(X)^2\big)^{-1} \|\grad \sfc(X) P\|_2^2 \|(\dot X, \dot P)\|^2,
  \end{align*}
  where the last step uses $\|(\dot X, \dot P)\|^2 = \|K\|_\frob^2 + \|V_p
  \sfactor(\alpha)^{1/2}\|_\frob^2$.
  The \cauchyschwarz{} inequality gives that $a_2^2 \leq 4a_1a_3$ and $b_2^2
  \leq 4b_1b_3$.
  It follows that
  \begin{align*}
    \|\bar K\|_\frob^2 + \|\bar V_p \sfactor(\alpha)^{1/2}\|_\frob^2 &= a_1 + b_1 + a_2 + b_2 + a_3 + b_3\\
                                                                     &\leq a_1 + b_1 + 2\big(\sqrt{a_1 a_3} + \sqrt{b_1 b_3}\big) + a_3 + b_3\\
                                                                     &\leq a_1 + b_1 + 2\sqrt{a_1 + b_1}\sqrt{a_3 + b_3} + a_3 + b_3\\
                                                                     &= \big(\sqrt{a_1 + b_1} + \sqrt{a_3 + b_3}\big)^2.
  \end{align*}
  The bound for $\|\hess \mfc(X, P)\|_\opnorm$ follows from this.
\end{proof}

The bound is still finite when $\sigma_r(X) = 0$ thanks to the term in $\alpha$.
In contrast, for a $\smooth{2}$ function on $\reals_r^{m \times n}$, the Riemannian Hessian can blow up as $\sigma_r(X)$ approaches zero.

\subsection{Local convergence with a \pl{}-type condition}\label{subsec:local-conv-pl}

The \polyakloja{} (\pl{}) condition is a standard assumption to derive local
convergence guarantees for optimization
algorithms~\citep{polyak1963gradient,attouch2010proximal,attouch2013convergence,bolte2014proximal,karimi2016linear}.

\begin{definition}\label{def:pl}
  Let $\optpoint$ be a local minimizer of a $\smooth{1}$ function $h \colon \nanifold \to
  \reals$ defined on a smooth manifold $\nanifold$.
  We say $h$ satisfies the \emph{\polyakloja{}} condition with constant
  $\plconstant > 0$ (also denoted \emph{$\plconstant$-\pl{}}) around $\optpoint$ if
  \begin{align}
    h(x) - h(\optpoint) \leq \frac{1}{2\plconstant}\|\grad h(x)\|^2
    \tag{\pl{}}
    \label{eq:pl}
  \end{align}
  for all $x$ in some neighborhood of $\optpoint$, where $\grad h$ is the
  gradient of $h$.
  Likewise, we say $h$ satisfies $\plconstant$-\pl{} around a set $\optimalset$
  of local minimizers if it does so around each $\optpoint \in \optimalset$.
\end{definition}

We seek to obtain the \pl{} condition for $\mfc$ via appropriate assumptions on $\sfc$.
However, it is \emph{a priori} not clear what to assume on $\sfc$ because $\boundedrank$ is nonsmooth.
In particular, \pl{} is not well defined for $\sfc$ around points of non-maximal rank.

To deal with these singular points, we resort to the local \morsebott{} property (which is equivalent to \pl{} for $\smooth{2}$ functions on manifolds), and we propose a version for $\sfc$ that accommodates the nonsmoothness of $\boundedrank$.
As a preliminary, we need the notion of tangent cone.

\paragraph{Tangent cones.}

At a point $X \in \boundedrank$, the \emph{tangent cone} to $\boundedrank$ is defined as
\begin{align*}
  \tangent_X \boundedrank = \bigg\{ \lim_{k \to \infty} \frac{X_k - X}{t_k} : X_k \in \boundedrank, t_k > 0 \text{ for all $k$},\, t_k \to 0 \text{ and the limit exists}\bigg\},
\end{align*}
which is a closed (not necessarily convex) cone.
A simple expression is known for this cone~\cite[Thm.~3.2]{schneider2015convergence}.
At a point $X$ of rank $s$, we have
\begin{equation}\label{eq:tgt-bounded}
  \begin{aligned}
    \tangent_X \reals_{\leq r}^{m \times n} = \bigg\{&
      \begin{bmatrix}
        U_s & U_{s, \perp}
      \end{bmatrix}
      \begin{bmatrix}
        A & B\\ C & D
      \end{bmatrix}
      \begin{bmatrix}
        V_s & V_{s, \perp}
      \end{bmatrix}^\top :\\
    &\qquad \text{$A, B, C, D$ are of appropriate sizes and $\rank(D) \leq r - s$} \bigg\},
  \end{aligned}
\end{equation}
where $U_s \in \stiefel(m, s)$ and $V_s \in \stiefel(n, s)$ span the column and
row spaces of $X$ respectively, and $U_{s, \perp}$ and $V_{s, \perp}$ are
orthonormal completions.
The tangent space $\tangent_X \reals_s^{m \times n}$ has the same set
description as~\eqref{eq:tgt-bounded} with the additional constraint that $D =
0$.

\paragraph{The \morsebott{} property.}

We now define the \morsebott{} property for $\mfc$, and a variant for $\sfc$ tailored to $\boundedrank$.
It requires $\sfc$ and $\mfc$ to be $\smooth{2}$.
We let $\optimalset_\sfc$ and $\optimalset_\mfc$ denote the sets that contain all the local minimizers of~\eqref{eq:problem} and~\eqref{eq:lifted-problem} respectively.

\begin{definition}\label{def:mb}
  Let $(X, P)$ be a local minimizer of $\mfc$.
  We say $\mfc$ satisfies the \emph{\morsebott{} property} at $(X, P)$ with
  constant $\plconstant > 0$ if
  \begin{equation}\label{eq:morse-bott}\tag{$\mb{}_\mfc$}
    \begin{aligned}
      & \optimalset_\mfc \textrm{ is a $\smooth{1}$ submanifold of $\desing$ around } (X, P), \textrm{ and } \\
      & \inner{(\dot X, \dot P)}{\hess \mfc(X, P)[\dot X, \dot P]} \geq \plconstant \|(\dot X, \dot P)\|^2 \qquad \forall (\dot X, \dot P) \in \normal_{(X, P)} \optimalset_\mfc,
    \end{aligned}
  \end{equation}
  where $\optimalset_\mfc$ is the set of local minimizers
  of~\eqref{eq:lifted-problem} and $\normal_{(X, P)} \optimalset_\mfc$ is the
  normal space to $\optimalset_\mfc$ at $(X, P)$.
  (This is with respect to the metric~\eqref{eq:inner-product} on $\tangent_{(X, P)}\desing$.)
\end{definition}

This condition is equivalent to \pl{} for $\smooth{2}$ functions, up to
arbitrarily small loss in the constant $\plconstant$~\citep{rebjock2024fast}.

It is now natural to ask: under what conditions on $\sfc$ is it the case that $\mfc = \sfc \circ \lift$ satisfies~\eqref{eq:morse-bott}?
We suggest one answer: a \morsebott{} type condition with an additional requirement on $\optimalset_\sfc$.
In Section~\ref{subsec:transitivity}, we show the implication $\mb{}_\sfc \Rightarrow \mb{}_\mfc$ (which could fail without the extra condition on $\optimalset_\sfc$).

\begin{definition} \label{def:MBf}
  Let $X$ be a local minimizer for~\eqref{eq:problem} of rank $s$.
  We say $\sfc$ satisfies the \emph{\morsebott{} type property} at $X$ with constant
  $\plconstant > 0$ if
  \begin{equation}\label{eq:mb-type}\tag{$\mb{}_\sfc$}
    \begin{aligned}
      & \optimalset_\sfc \text{ is a $\smooth{1}$ submanifold of $\reals^{m \times n}$ included in $\reals_s^{m \times n}$ around $X$},
        \textrm{ and } \\
      & \inner{\dot X}{\hess \sfc(X)[\dot X]} \geq \plconstant \|\dot X\|_\frob^2 \qquad \forall \dot X \in \normal_X \optimalset_\sfc \cap \tangent_X \boundedrank,
    \end{aligned}
  \end{equation}
  where $\optimalset_\sfc$ is the set of local minimizers of~\eqref{eq:problem} and
  $\normal_{X} \optimalset_\sfc$ is the normal space to $\optimalset_\sfc$ at
  $X$.
  (This is with respect to the Frobenius inner product on $\reals^{m \times n}$.)
\end{definition}

Note the two differences between~\eqref{eq:morse-bott} and~\eqref{eq:mb-type}.
For the latter, \emph{(i)} we require the solution set $\optimalset_\sfc$ to be
included in a constant rank stratum $\reals_s^{m \times n}$, and \emph{(ii)} the
inner product lower bound needs to hold only for matrices $\dot X$ in the
tangent cone at $X$.
This second fact makes the condition depend only on the values that $\sfc$ takes
in $\boundedrank$ (as opposed to depending also on the values of $\sfc$ in a neighborhood of $\boundedrank$ in $\reals^{m \times n}$).
More explicitly, if two functions $\sfc_1$ and $\sfc_2$ coincide on
$\boundedrank$ then the quantities $\inner{\dot X}{\hess \sfc_i(X)[\dot X]}$ are
the same when $\dot X \in \tangent_X \boundedrank$.
(This can be seen using that $\boundedrank$ is geometrically
derivable~\citep[Prop.~2]{o2004limits}.)
This is arguably a desirable property, as it makes our assumption intrinsic.

\subsection{Transitivity of a \pl{}-type condition}\label{subsec:transitivity}

In this section we prove that~\eqref{eq:mb-type} implies~\eqref{eq:morse-bott}.
Theorem~\ref{th:local} is then a direct consequence.
We study the maximal rank stratum and lower-rank strata separately, with different arguments.

\paragraph{Maximal rank stratum.}

We first examine \pl{} on constant rank strata and deduce \pl{} transitivity for the maximal rank stratum.
In this part we only require $\sfc$ (and hence $\mfc$) to be continuously differentiable ($\smooth{1}$).
We derive below an expression for the tangent spaces of constant rank strata, which we need for the analysis.

\begin{lemma}\label{lemma:tgt-stratum}
  Given $s \in \{0, \dots, r\}$, let $(X, P) \in \lift^{-1}(\reals_s^{m \times n})$ have representation $(U, \Sigma, V)$, and let $U_s \in \reals^{m \times s}$ hold the first $s$ columns of $U$.
  Then
  \begin{equation}\label{eq:tgt-stratum}
    \begin{aligned}
      \tangent_{(X, P)} \lift^{-1}(\reals_s^{m \times n})
      &= \Big\{\pig(
        \begin{bmatrix}
          K_s & U_s K_{sr}
        \end{bmatrix}
        V^\top + U \Sigma V_p^\top, -V_p V^\top - V V_p^\top\pig)\\
      &\qquad\quad: K_s \in \reals^{m \times s}, K_{sr} \in \reals^{s \times (r - s)}, V_p \in \reals^{n \times r}, V\transpose V_p = \zeros
        \Big\}.
    \end{aligned}
  \end{equation}
\end{lemma}
\begin{proof}
  Proposition~\ref{prop:submanifold-preimage} provides $\tangent_{(X, P)}
  \lift^{-1}(\reals_s^{m \times n}) = \big\{ (\dot X, \dot P) \in \tangent_{(X,
    P)} \desing : \dot X \in \tangent_X \reals_s^{m \times n} \big\}$.
  The right-hand side of~\eqref{eq:tgt-stratum} is included in this set: this
  follows from Proposition~\ref{prop:tangent-space} and~\eqref{eq:tgt-bounded}
  (with $D = \zeros$).
  Counting degrees of freedom and Proposition~\ref{prop:submanifold-preimage}
  give that the two spaces have the same dimension $(m + n - s)s + (n - r)(r -
  s)$, hence the equality.
\end{proof}

\begin{proposition}\label{prop:pl-stratum}
  Given $s \in \{0, \dots, r\}$, we let $\sfc_s$ and $\mfc_s$ be the
  restrictions of $\sfc$ and $\mfc$ to the respective rank $s$ strata
  $\reals_s^{m \times n}$ and $\lift^{-1}(\reals_s^{m \times n})$.
  If $\sfc_s$ is $\plconstant$-\eqref{eq:pl} around a local minimizer $X \in
  \reals_s^{m \times n}$ then $\mfc_s$ is $\frac{\sigma_s(X)^2}{2\alpha +
    \sigma_s(X)^2}\plconstant'$-\eqref{eq:pl} around $\lift^{-1}(X)$ for all
  $\plconstant' < \plconstant$.
\end{proposition}
\begin{proof}
  Let $\mathcal{U}$ be a neighborhood of $X$ in $\reals_s^{m \times n}$ where $\sfc_s$ satisfies~\eqref{eq:pl}, and define $\mathcal{V} = \lift^{-1}(\mathcal{U})$.
  Let $(Y, P) \in \mathcal{V}$ have representation $(U, \Sigma, V)$.
  Let also $U_s$ and $V_s$ hold the first $s$ columns of $U$ and $V$, and $V_{sr}$ hold the last $r - s$ columns of $V$.
  Proposition~\ref{prop:gradient} gives that $\grad \mfc(Y, P)$ has representation $\big(\big[\begin{matrix} K_s & K_{sr} \end{matrix}\big], V_p\big)$, where $K_s = \grad \sfc(Y) V_s$, $K_{sr} = \grad \sfc(Y) V_{sr}$ and $V_p = P \grad \sfc(Y)\transpose U \Sigma \sfactor(\alpha)^{-1}$.
  From the identity
  \begin{equation*}
    \grad \mfc_s(Y, P) = \proj_{\tangent_{(Y, P)} \lift^{-1}(\reals_s^{m \times n})} \grad \mfc(Y, P),
  \end{equation*}
  and from the expression of the tangent space $\tangent_{(Y, P)}\lift^{-1}(\reals_s^{m \times n})$ given in Lemma~\ref{lemma:tgt-stratum},
  we deduce that $\grad \mfc_s(Y, P)$ has representation $\big(\big[\begin{matrix} K_s & U_s^{} U_s\transpose K_{sr} \end{matrix}\big], V_p\big)$.
  It follows that for all $(Y, P) \in \mathcal{V}$, the squared norm of the gradient (Proposition~\ref{prop:inner-product}) satisfies
  \begin{align*}
    \|\grad \mfc_s(Y, P)\|^2 &= \|\grad \sfc(Y) V_s\|_\frob^2 + \|U_s^{} U_s\transpose \grad \sfc(Y) V_{sr}\|_\frob^2 + \|\Sigma \sfactor(\alpha)^{-1/2} U\transpose \grad \sfc(Y) P\|_\frob^2 \\
                             &\geq \sigma_s\big(\Sigma^2 \sfactor(\alpha)^{-1}\big) \Big( \|\grad \sfc(Y) V_s\|_\frob^2 + \|U_s^{} U_s\transpose \grad \sfc(Y) V_{sr}\|_\frob^2 + \|U_s^{} U_s\transpose \grad \sfc(Y) P\|_\frob^2 \Big) \\
                             &= \frac{\sigma_s(Y)^2}{2\alpha + \sigma_s(Y)^2} \|\grad \sfc_s(Y)\|_\frob^2.
  \end{align*}
  For the last term in the inequality, we used
  \begin{equation*}
    \|U\transpose \grad \sfc(Y) P\|_\frob = \|UU\transpose \grad \sfc(Y) P\|_\frob \geq \|U_s^{} U_s\transpose \grad \sfc(Y) P\|_\frob.
  \end{equation*}
  For the last equality, we used the identity
  \begin{equation*}
    \grad \sfc_s(Y) = \proj_{\tangent_Y \reals_s^{m \times n}} \grad \sfc(Y)
  \end{equation*}
  and the expression of the tangent space $\tangent_Y \reals_s^{m \times n}$ given in~\eqref{eq:tgt-bounded} (with $D = \zeros$, and using $P = I - VV^\top$).
  The \pl{} condition then follows from $\|\grad \sfc_s(Y)\|_\frob^2 \geq 2\plconstant \big(\sfc(Y) - \sfc(X)\big)$ and from continuity of the map $Y \mapsto \sigma_s(Y)$.
\end{proof}

It is informative to particularize Proposition~\ref{prop:pl-stratum} to $s = r$.
Indeed, the maximal-rank stratum of $\desing$ is open (see Remark~\ref{rmk:max-fiber-diffeomorphic}), so the restriction $\mfc_r$ coincides with $\mfc$ on $\lift^{-1}(\reals_r^{m \times n})$; in particular, their gradients (hence gradient norms) agree there.
Assuming that $\sfc_r$ satisfies~\eqref{eq:pl} around a point $X \in \reals_r^{m \times n}$, Proposition~\ref{prop:pl-stratum} implies that $\mfc_r$ satisfies~\eqref{eq:pl} around $\lift^{-1}(X)$, and therefore so does $\mfc$ on the same neighborhood.
However, this reasoning fails for lower-rank strata because they are not open, so we handle those next by a different argument.

\paragraph{Non-maximal rank strata.}

For lower-rank strata, we work with the \morsebott{} property, as defined in Section~\ref{subsec:local-conv-pl}.
We first prove a technical lemma and then state the main result of this section.

\begin{lemma}\label{lemma:l2-zero}
  Let $\optimalset$ be an embedded submanifold of $\reals_s^{m \times n}$.
  Let $(X, P) \in \lift^{-1}(\optimalset)$ and $(\dot X, \dot P) \in \tangent_{(X, P)} \desing$ have representations $(U, \Sigma, V)$ and $(K, V_\perp L)$.
  If $(\dot X, \dot P)$ is in $\normal_{(X, P)} \lift^{-1}(\optimalset)$ then $L = (\begin{matrix} L_1 & \zeros_{(n - r) \times (r - s)} \end{matrix})$ for some $L_1 \in \reals^{(n - r) \times s}$.
\end{lemma}
\begin{proof}
  We know from Proposition~\ref{prop:submanifold-preimage} that
  $\lift^{-1}(\optimalset)$ is an embedded submanifold of $\desing$.
  The inclusion $\lift^{-1}(X) \subseteq \lift^{-1}(\optimalset)$ implies that
  $\tangent_{(X, P)} \lift^{-1}(X) \subseteq \tangent_{(X, P)}
  \lift^{-1}(\optimalset)$.
  From Propositions~\ref{prop:tangent-space} and~\ref{prop:submanifold-preimage}
  we also find that
  \begin{align*}
    \tangent_{(X, P)} \lift^{-1}(X) = \big\{ (\zeros, -V_\perp \Gamma V^\top - V \Gamma\transpose V_\perp^\top) : \Gamma = (\begin{matrix} \zeros_{(n - r) \times s} & \Gamma_2 \end{matrix}), \Gamma_2 \in \reals^{(n - r) \times (r - s)} \big\}.
  \end{align*}
  It follows that the last $r - s$ columns of $L$ are indeed zero when $(\dot
  X, \dot P)$ is orthogonal to $\tangent_{(X, P)} \lift^{-1}(\optimalset)$.
\end{proof}

\begin{theorem}\label{th:pl-rank-defficient}
  Suppose $\sfc$ and $\mfc$ are $\smooth{2}$.
  Let $X \in \optimalset_\sfc$ be a local minimizer for~\eqref{eq:problem} of rank
  $s \leq r$.
  Suppose~\eqref{eq:mb-type} holds at $X$ with constant $\plconstant$.
  Then $\mfc$ satisfies~\eqref{eq:morse-bott} on the fiber $\lift^{-1}(X)$ with
  constant $\frac{\sigma_s(X)^2}{2\alpha + \sigma_s(X)^2} \plconstant$.
\end{theorem}
\begin{proof}
  Fix $P$ such that $(X, P)$ is in $\lift^{-1}(X)$ and let $(U, \Sigma, V)$ be a
  representation (Definition~\ref{def:pt-repr}) of $(X, P)$.
  Consider first the case $s=r$.
  Since the maximal-rank stratum $\reals^{m\times n}_r$ is an open subset of $\boundedrank$ (and a smooth embedded submanifold of $\reals^{m\times n}$), the condition~\eqref{eq:mb-type} coincides with the standard \morsebott{} property~\cite{rebjock2024fast}.
  Hence the case $s=r$ follows directly from Proposition~\ref{prop:pl-stratum} and the equivalence between \pl{} and \mb{} \cite[Cor.~2.17]{rebjock2024fast}.
  \TODOF{NB: About the first part of the claim though: Not quite.
    Proposition~\ref{prop:pl-stratum} loses a bit in the PL constant, so we
    would also lose a bit in the MB constant.
    Modify the theorem statement to have a $\plconstant' < \plconstant$?
    QR: We can take $\plconstant'$ as close to $\plconstant$ as desired.
    This makes the MB constant $\geq \plconstant$.}
  We now assume $s < r$.
  The gradient $\grad \sfc(X)$ is zero because $X$ is a stationary point of rank $s < r$~\cite[Cor.~3.4]{schneider2015convergence}.
  Proposition~\ref{prop:submanifold-preimage}, together with the definition of~\eqref{eq:mb-type}, implies that $\lift^{-1}(\optimalset_\sfc)$ is a submanifold locally around $(X, P)$ and
  \begin{align}\label{eq:tgt-fiber-upstairs}
    \tangent_{(X, P)} \lift^{-1}(\optimalset_\sfc) = \big\{ (\dot X, \dot P) \in \tangent_{(X, P)} \desing : \dot X \in \tangent_X \optimalset_\sfc \big\}.
  \end{align}
  We study the Hessian at that point to prove that $\mfc$
  satisfies~\eqref{eq:morse-bott}.
  Let $(\dot X, \dot P) \in \tangent_{(X, P)}\desing$ be in the normal space of
  $\lift^{-1}(\optimalset_\sfc)$ at $(X, P)$.
  Let $\dot X_p$ and $\dot X_t$ be the orthogonal projections (for the Frobenius
  inner product) of $\dot X$ onto $\normal_X \optimalset_\sfc$ and $\tangent_X
  \optimalset_\sfc$ respectively.
  The two following properties hold:
  \begin{itemize}
  \item $\dot X_t$ is in the kernel of $\hess \sfc(X)$.
    Indeed, decompose $\dot X = \dot X_p + \dot X_t$.
    Since $\optimalset_\sfc$ is $\smooth{1}$ there exists a $\smooth{1}$ curve
    $c \colon \reals \to \optimalset_\sfc$ such that $c(0) = X$ and $c'(0) =
    \dot X_t$.
    For all $t$ near zero we have $\nabla f(c(t)) = 0$ since each point in
    $\optimalset_\sfc$ near $X$ is a local minimizer of~\eqref{eq:problem} with
    rank $s < r$, owing to $\optimalset_\sfc$ being locally included in
    $\mathbb{R}^{m \times n}_s$.
    We deduce $0 = \left. \frac{\diff}{\diff t} \nabla f(c(t)) \right|_{t = 0} =
    \nabla^2 f(X)[\dot X_t]$.
  \item $\dot X_p$ is an element of $\tangent_X \boundedrank$.
    This is because $\dot X \in \tangent_X \boundedrank$ and $\dot X_t \in
    \tangent_X \reals_s^{m \times n}$.
    The special structure of $\tangent_X \boundedrank$ given
    in~\eqref{eq:tgt-bounded} ensures that the difference $\dot X_p = \dot X -
    \dot X_t$ is also in $\tangent_X \boundedrank$.
  \end{itemize}
  Corollary~\ref{cor:hess-inner-product}, the equality $\grad \sfc(X) = \zeros$, the two points above, and~\eqref{eq:mb-type} give
  \begin{align}\label{eq:inner-bigger-xp}
    \inner{(\dot X, \dot P)}{\hess \mfc(X, P)[\dot X, \dot P]} = \inner{\dot X}{\hess \sfc(X)[\dot X]} = \inner{\dot X_p}{\hess \sfc(X)[\dot X_p]} \geq \plconstant \|\dot X_p\|_\frob^2.
  \end{align}
  We now find a lower bound for $\|\dot X_p\|_\frob^2$ with respect to $\|(\dot X, \dot P)\|^2$.
  To do this, we first characterize $\normal_{(X, P)} \lift^{-1}(\optimalset_\sfc)$.
  Let $(K, V_\perp L)$ be a representation (Definition~\ref{def:tgt-repr}) of $(\dot X, \dot P)$.
  Lemma~\ref{lemma:l2-zero} gives that $L = (\begin{matrix} L_1 & \zeros_{(n - r) \times (r - s)} \end{matrix})$ for some $L_1 \in \reals^{(n - r) \times s}$.
  Let $\Sigma_s$ denote the top left $s \times s$ submatrix of $\Sigma$ (which is positive definite in particular).
  Owing to Proposition~\ref{prop:inner-product}, the inclusion $(\dot X, \dot P) \in \normal_{(X, P)} \lift^{-1}(\optimalset_\sfc)$ holds if and only if
  \begin{align}\label{eq:tgt-inner-zero}
    \inner{K_t}{K} + \inner{L_{t1}}{L_1 (2 \alpha I + \Sigma_s^2)} = 0
  \end{align}
  for all $(\dot X_t, \dot P_t) \in \tangent_{(X, P)} \lift^{-1}(\optimalset_\sfc)$ with representation $(K_t, L_t)$, and where $L_{t1} \in \reals^{(n - r) \times s}$ is composed of the first $s$ columns of $L_t$.

  Define the matrix $\dot X_n = K_n V^\top + U \Sigma L_n\transpose V_\perp^\top$, where $K_n = K$, $L_n = (\begin{matrix}L_{n1} & \zeros_{(n - r) \times (r - s)}\end{matrix})$, and $L_{n1} = L_1 \Sigma_s^{-2} (2 \alpha I + \Sigma_s^2)$.
  Then $\dot X_n$ satisfies (see Proposition~\ref{prop:inner-product})
  \begin{align}\label{eq:eqxn}
    \inner{\dot X}{\dot X_n} = \inner{K}{K_n} + \inner{L}{L_n \Sigma^2} = \|K\|_\frob^2 + \inner{L}{L \sfactor(\alpha)} = \|(\dot X, \dot P)\|^2,
  \end{align}
  and its norm is bounded as
  \begin{align}\label{eq:ineqxn}
    \|\dot X_n\|_\frob^2 = \|K\|_\frob^2 + \|L_1 \Sigma_s^{-1} (2 \alpha I + \Sigma_s^2)\|_\frob^2 \leq \frac{2 \alpha + \sigma_s(X)^2}{\sigma_s(X)^2} \|(\dot X, \dot P)\|^2.
  \end{align}
  We also find that $\dot X_n$ is in $\normal_X\optimalset_\sfc$.
  Indeed, given a tangent vector $X_t \in \tangent_X \optimalset_\sfc$ with representation $(K_t, L_t)$, we compute
  \begin{align*}
    \inner{\dot X_t}{\dot X_n} = \inner{K_t}{K_n} + \inner{L_{t1}}{L_{n1} \Sigma_s^2} = \inner{K_t}{K} + \inner{L_{t1}}{L_1 (2 \alpha I + \Sigma_s^2)} = 0,
  \end{align*}
  where the last equality comes from~\eqref{eq:tgt-inner-zero}.
  Since $\dot X_p$ is the projection of $\dot X$ onto $\normal_X \optimalset_\sfc$, we deduce that the norm of $\dot X_p$ is at least that of the projection of $\dot X$ onto $\{\lambda \dot X_n : \lambda \in \reals\}$.
  From this,~\eqref{eq:eqxn} and~\eqref{eq:ineqxn}, it follows that
  \begin{align}\label{eq:lower-bound-norm-xp}
    \|\dot X_p\|^2 \geq \frac{\inner{\dot X}{\dot X_n}^2}{\|\dot X_n\|_\frob^2} = \frac{\|(\dot X, \dot P)\|^4}{\|\dot X_n\|_\frob^2} \geq \frac{\sigma_s(X)^2}{2 \alpha + \sigma_s(X)^2} \|(\dot X, \dot P)\|^2.
  \end{align}
  Inequalities~\eqref{eq:inner-bigger-xp} and~\eqref{eq:lower-bound-norm-xp} together imply property~\eqref{eq:morse-bott}.
\end{proof}

Theorem~\ref{th:local} is then a direct consequence.

\begin{proof}[Proof of Theorem~\ref{th:local}]
  Theorem~\ref{th:pl-rank-defficient} gives the implication~\eqref{eq:mb-type} $\Rightarrow$~\eqref{eq:morse-bott}, and \mb{} yields \pl{}~\citep{rebjock2024fast}.
\end{proof}
When $\mfc$ satisfies \mb{} (and equivalently \pl{}), we can deduce local convergence results for standard algorithms, as described below Theorem~\ref{th:local}.

\begin{remark}
  Theorem~\ref{th:pl-rank-defficient} assumes that all the points of
  $\optimalset$ around $X$ have the same rank.
  At first sight, this restriction may appear unnecessarily strong.
  But without it we cannot guarantee that $\lift^{-1}(\optimalset)$ is (locally)
  a submanifold (see Remark~\ref{rmk:cst-rank}).
  Yet, this set needs to be a submanifold for \pl{} to hold~\cite[Thm.~2.16]{rebjock2024fast}.
\end{remark}

\begin{remark}\label{rmk:feasibility-of-mb-type}
  Our local theory encompasses all existing guarantees on the maximal-rank stratum.
  Beyond that, it can handle situations where the minimizer lies on a lower-rank stratum, which we discuss below.

  An important special case is when the minimizer of~\eqref{eq:problem} is isolated.
  This corresponds to~\eqref{eq:mb-type} where the set of minimizers $\optimalset_\sfc$ is a singleton.
  In this setting, we guarantee that \morsebott{} holds for~\eqref{eq:lifted-problem}.

  This situation does arise in practical applications---for example, in matrix sensing problems.
  Given $X^* \in \boundedrank$, a linear operator $\mathcal{A} \colon \reals^{m \times n} \to \reals^p$, and $y = \mathcal{A}[X^*]$, consider the objective $\sfc \colon X \mapsto \|\mathcal{A}[X] - y\|_\frob^2$ on $\boundedrank$.
  It satisfies~\eqref{eq:mb-type} when $\mathcal{A}^*\mathcal{A}$ is full rank.
  This occurs, for instance, when $\mathcal{A}$ extracts i.i.d.\ Gaussian measurements and $p$ is large enough.
  A sufficient condition is $p \ge mn$, which is atypical in applications, but numerical evidence shows that fewer measurements often suffice.
  This formulation also covers the matrix approximation problem.

  Our theory extends beyond the singleton case under a structural requirement on $\optimalset_\sfc$, as specified in the definition of~\eqref{eq:mb-type}.
  This requirement is restrictive and limits the range of applications.

  In particular, for matrix completion,~\eqref{eq:mb-type} generally fails at minimizers of non-maximal rank.
  This is because rank-1 matrices lie in the kernel of the mask operator (unless it is complete), so if a non-maximal-rank matrix fits the data perfectly, arbitrarily close higher-rank matrices can do so as well.
\end{remark}

\section{Numerical experiments}\label{sec:experiments}

In this section we run matrix completion experiments to compare the LR parameterization, fixed-rank optimization, and the desingularization (with three
metric parameters: $\alpha = 1/20$, $\alpha = 1/2$ and $\alpha = 5$).
Matrix completion does not satisfy the assumptions of Section~\ref{sec:conv-guarantees}---namely, compact sublevel sets and the \morsebott{} condition of Theorems~\ref{th:global} and~\ref{th:local} (see Remark~\ref{rmk:feasibility-of-mb-type}).
Nevertheless, it remains an important problem in low-rank optimization, which motivates its use here for illustration.

Given $A \in \reals^{m \times n}$ and a mask $\Omega \in \{0, 1\}^{m \times n}$, the cost function $\sfc \colon \reals^{m \times n} \to \reals$ is defined as $\sfc(X) = \frac{1}{2}\|(X - A) \odot \Omega\|_{\frob}^2$, where $\odot$ denotes the Hadamard product.
The (Euclidean) gradient and Hessian are $\grad \sfc(X) = (X - A) \odot \Omega$ and $\hess \sfc(X)[\dot X] = \dot X \odot \Omega$.
Both are sparse when $\Omega$ is sparse: this is exploited in the code.

We set the dimensions as $m = n = 5000$ in all the experiments.
The rank $r^*$ of $A$ is comparatively small (see details below).
We construct $A = U_A^{} \Sigma_A^{} V_A^\top$ for some matrix of singular
values $\Sigma_A \in \reals^{r^* \times r^*}$, and where $U_A \in \stiefel(m,
r^*)$ and $V_A \in \stiefel(n, r^*)$ are sampled uniformly at random.
The mask $\Omega$ is sparse with an oversampling factor of $5$, meaning that it
contains $5(m + n - r) r$ non-zero entries.
They are chosen uniformly at random.

We interface the implementation with the Matlab toolbox
Manopt~\citep{boumal2014manopt}.
For all geometries, we run a (Riemannian) trust-region
algorithm~\citep{absil2007trust} with the truncated conjugate gradient
subproblem solver.\footnote{We set the power for the residual stopping
  criterion~\citep[Eq.~(10)]{absil2007trust} as $\theta = \sqrt{2} - 1$ (instead
  of the usual default $\theta = 1$).
  Indeed, it appears that setting $\theta < 1$ helps convergence for all the
  methods.}
Evaluating the cost function, the gradient, and Hessian vector products takes
similar times for the three geometries.
For a given experiment, all the algorithms start from the same initial random
point in $\boundedrank$ (with balanced factors for the LR parameterization).
We pick the initialization close to zero (compared to $A$) because we found that
it significantly improves performance for all.
Specifically, we sample $U \in \stiefel(m, r), V \in \stiefel(n, r), \sigma_1,
\dots, \sigma_r \in \interval{0}{10^{-3}}$ all independently and uniformly at
random, and then set the initial point as $U \diag(\sigma_1, \dots, \sigma_r)
V^\top$.


For each experiment we report two figures.
Their vertical axes display the cost function value.
The horizontal axis of the first figure is the number of outer iterations of
trust-regions.
The horizontal axis of the second figure is the time in seconds.


\paragraph{Overestimation of the rank.}

We let the target matrix $A$ have rank $r^* = 10$ and optimize with $r = 20$.
The 10 non-zero singular values of $A$ are sampled uniformly at random in the
interval $\interval{\frac{1}{2}}{1}$.
Figure~\ref{fig:exp-1} reports the results.
\begin{figure}
  \centering
  \includegraphics[width=0.49\textwidth]{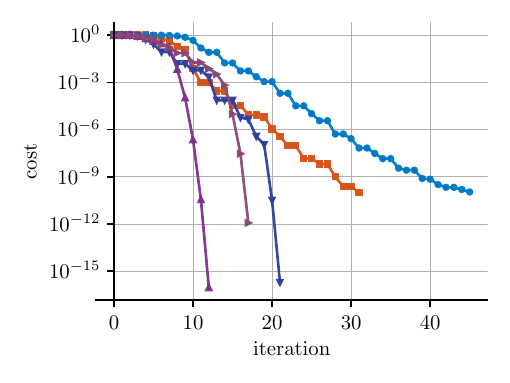}
  \hfill
  \includegraphics[width=0.49\textwidth]{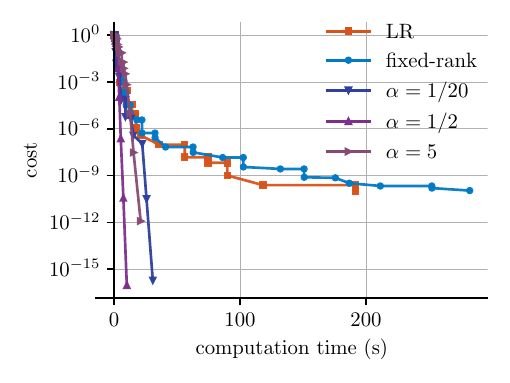}
  \caption{Rank overestimation $r > r^*$ and small condition number. Our geometry is tested with various choices of $\alpha$ for the Riemannian metric~\eqref{eq:inner-product}, compared to the $LR\transpose$ parameterization and optimization on the fixed-rank manifold.}
  \label{fig:exp-1}
\end{figure}
In this scenario the convergence is particularly fast for the desingularization.

\paragraph{Exponential decay.}

We let $A$ have rank $r^* = 20$ with singular values that follow an exponential decay as $\sigma_i(A) = 0.9^{i - 1}$ for $i = 1, \dots, r^*$.
Figure~\ref{fig:exp-2} displays the results with an optimization rank $r = r^* = 20$.
\begin{figure}[h]
  \centering
  \includegraphics[width=0.49\textwidth]{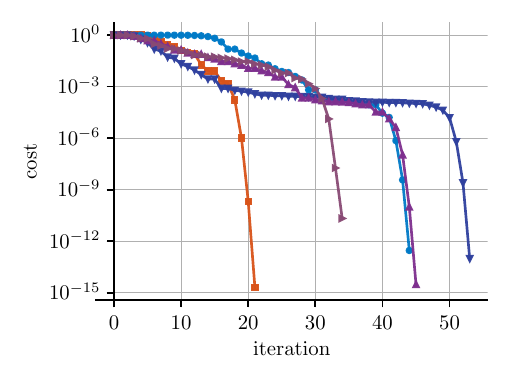}
  \hfill
  \includegraphics[width=0.49\textwidth]{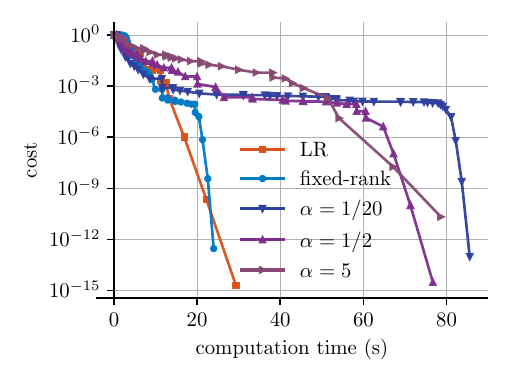}
  \caption{Exact rank $r = r^*$ and exponential decay of target singular values.}
  \label{fig:exp-2}
\end{figure}
In this setting both LR and fixed-rank optimization are faster than the desingularization.
Figure~\ref{fig:exp-3} reports the results for the exact same experiment but with an optimization rank of $r = 30 > r^*$.
\begin{figure}[h]
  \centering
  \includegraphics[width=0.49\textwidth]{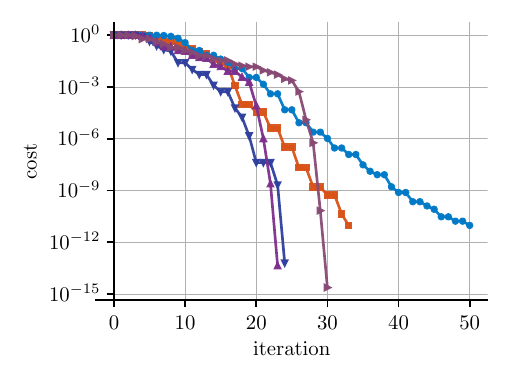}
  \hfill
  \includegraphics[width=0.49\textwidth]{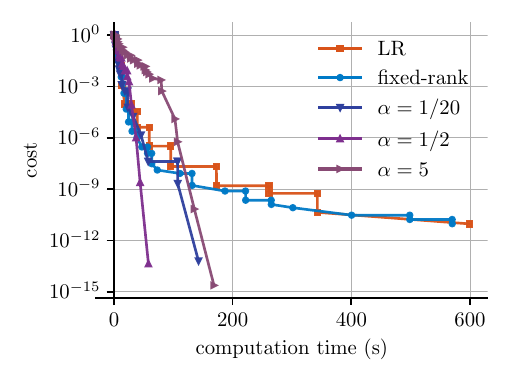}
  \caption{Rank overestimation $r > r^*$ and exponential decay of target singular values.}
  \label{fig:exp-3}
\end{figure}
This is detrimental for LR and fixed-rank optimization but the desingularization is robust to this rank overestimation.

\paragraph{Other solvers for matrix completion.}

This paper deals with general algorithms for bounded-rank optimization.
The numerical experiments compare general algorithms on a specific task, namely,
matrix completion.
There exist many solvers tailored for the latter.
For example, the algorithms proposed
in~\citep{kummerle2021scalable,bauch2021rank,zilber2022gnmr} exhibit
state-of-the-art performances in practice.
We experimented with GNMR~\citep{zilber2022gnmr} and empirically observed the
following.
This algorithm is particularly efficient to recover the matrix $A$, even in hard
settings where the oversampling factor is close to $1$, or when the condition
number is large.
However, iterations are expensive and we found that GNMR is slower than the
trust-region method.
We do not include these results for this reason.

\section{Perspectives}


We close with a few directions for future research.
\begin{itemize}
\item The desingularization is particularly relevant when the true (or
  numerical) rank $r^*$ of the target matrix is strictly less than the optimization
  rank $r$.
  This is especially convenient when $r^*$ is not known because it
  allows the user to set $r$ as a conservative upper bound of $r^*$.
  In that scenario, we may want to nudge the algorithm towards solutions of rank
  less than $r$.
  This could be achieved with a nuclear norm penalty~\citep{mishra2013low}, thus
  mixing both a hard and a soft low-rank prior.
\item Riemannian optimization theory breaks for algorithms running over
  $\reals_r^{m \times n}$ when the sequence converges to a point of rank
  strictly less than $r$ (because it lies outside of the manifold).
  Notwithstanding, even in that setting, we sometimes observe fast local
  convergence.
  \cite{luo2024tensor} explain that behavior for gradient descent on specific
  cost functions.
  In order to generalize such results, it may be possible to use the
  desingularization with the family of Riemannian metrics we propose, and to study algorithms in the limit where the metric
  parameter $\alpha$ goes to zero.
  Indeed, $\reals_r^{m \times n}$ (with its usual metric) is isometric to the maximal rank stratum of
  $\desing$ when $\alpha = 0$ (see Remark~\ref{rmk:max-fiber-isometric}).
\item Our local convergence results focus on the \polyakloja{} condition.
  They may generalize to the broader class of \loja{} inequalities.
  Especially, one might ask: does $\mfc$ satisfy a \loja{} inequality if $\sfc$
  satisfies an adequately modified version of that inequality for
  $\boundedrank$?
  A starting point could be~\citep{schneider2015convergence}.
\end{itemize}

\section*{Acknowledgments}

We thank Yuetian Luo and Anru Zhang for offering detailed pointers to relevant literature,
as well as Eitan Levin, Joe Kileel, Bart Vandereycken and Guillaume Olikier for numerous discussions.

\bibliographystyle{plainnat}
\bibliography{bibliography}

\appendix
\section{A note about regularizing the LR parameterization}\label{sec:regularizers}

It is common to regularize the LR parameterization.
The cost function becomes $\mfc(L, R) = \sfc(LR^\top) + \lambda \rho(L, R)$, where $\sfc$ is the function to minimize over $\boundedrank$, $\rho \colon \reals^{m \times r} \times \reals^{n \times r} \to \reals$ is the regularizer, and $\lambda > 0$ is the regularization intensity.

One possible choice is $\rho(L, R) = \|L\|_\frob^2 + \|R\|_\frob^2$.
This promotes solutions with small nuclear norm~\cite[Lem.~8]{srebro2004learning}.
While this is useful in several contexts, we do not entertain it here because it changes the minimizers of the cost function.

As introduced in~\citep{tu2016low}, the regularizer $\rho(L, R) = \|L^\top L - R^\top R\|_\frob^2$ promotes balanced factors.
It does \emph{not} change the global minimizers of the problem because the LR factors of a solution can always be balanced.
This regularizer has other nice properties, as described in~\citep{zhu2018global}.
However, in general the regularized cost $\mfc$ does not satisfy \pl{} at points of non-maximal rank.
This can happen even for simple costs such as
\begin{align*}
  (L, R) \mapsto \frac{1}{2}\|LR^\top - A\|_\frob^2 + \frac{\lambda}{4}\|L^\top L - R^\top R\|_\frob^2
\end{align*}
with $\rank A < r$.
For example, when $m = n = r = 1$ and $A = \zeros$, the cost becomes $(x, y) \mapsto \frac{1}{2}x^2 y^2 + \frac{\lambda}{4}(x^2 - y^2)^2$, and the Hessian is zero at $(0, 0)$.

\section{Second-order derivatives}\label{sec:appendix-second-order-derivatives}

The expression of the Hessian involves the derivative of the projector onto tangent spaces (see~\citep{absil2013extrinsic}, \citep[\S5.11]{boumal2020introduction}).
To find it we define four auxiliary functions:
\begin{equation}\label{eq:qs}
  \begin{aligned}
    &\qa(X) = (2\alpha I + XX^\top)^{-1}, && && \qb(X) = XX\transpose \qa(X) = I - 2 \alpha \qa(X),\\
    &\qc(X) = \qa(X)X, && && \qd(X) = (2\alpha I + X\transpose X)^{-1}.
  \end{aligned}
\end{equation}
With $X = U \Sigma V^\top$ an SVD, the definition of
$\sfactor(\alpha)$~\eqref{eq:sfactor} and direct computations give
\begin{equation}\label{eq:qs-usv}
  \begin{aligned}
    &\qa(X) = U \sfactor(\alpha)^{-1} U^\top + \frac{1}{2\alpha} U_\perp^{} U_\perp^\top, && && \qb(X) = U\Sigma^2 \sfactor(\alpha)^{-1} U^\top,\\
    &\qc(X) = U \sfactor(\alpha)^{-1} \Sigma V^\top, && && \qd(X) = V \sfactor(\alpha)^{-1} V^\top + \frac{1}{2\alpha} P.
  \end{aligned}
\end{equation}
Let $\sympart(A) = (A + A^\top)/2$ denote the symmetric part of a matrix $A$.
Using the functions defined in~\eqref{eq:qs}, we first express the orthogonal
projector (Proposition~\ref{prop:proj}) as a function of the point $(X, P)$
rather than its representation $(U, \Sigma, V)$.

\begin{lemma}\label{lemma:proj-x-p}
  Let $(X, P) \in \desing$ and $(Y, Z) \in \embeddingspace$.
  Then $(\dot X, \dot P) = \proj_{(X, P)}(Y, Z)$ satisfies
  \begin{align*}
    \dot X & = Y(I - P) + \qb(X)YP - 2 \alpha \qc(X) Z P, \\
    \dot P & = 2\sympart\!\Big(-PY\transpose \qc(X) + 2 \alpha P Z \qd(X) - P Z P\Big).
  \end{align*}
\end{lemma}
\begin{proof}
  From Proposition~\ref{prop:proj} and the tangent vectors representation
  (Definition~\ref{def:tgt-repr}) we find
  \begin{equation}\label{eq:proj-fxp-intermediate}
    \begin{aligned}
      \dot X &= Y (I - P) + U \Sigma \sfactor(\alpha)^{-1} \big(\Sigma U\transpose Y - 2 \alpha V\transpose Z\big) P\\
      \text{and }
      \dot P &= -2\sympart\!\Big(P\big(Y\transpose U \Sigma - 2 \alpha Z V\big) \sfactor(\alpha)^{-1} V^\top\Big).
    \end{aligned}
  \end{equation}
  Now to express $(\dot X, \dot P)$ as a function of $(X, P)$ rather than $(U,
  \Sigma, V)$,
  plug~\eqref{eq:qs-usv} into~\eqref{eq:proj-fxp-intermediate}.
\end{proof}

We can now scrutinize the differential of the orthogonal projection.
Given $(X, P) \in \desing$ and $(\dot X, \dot P) \in \tangent_{(X, P)}\desing$,
we define $\projdiff \colon \embeddingspace \to \embeddingspace$ as
\begin{align}\label{eq:def-projdiff}
  \projdiff(Y, Z) = \D\!\left( (X, P) \mapsto \proj_{(X, P)}(Y, Z)\right)\!(X, P)[\dot X, \dot P].
\end{align}
We explicitly compute this object because it plays a role in the Hessian of
$\mfc$.

\begin{lemma}\label{lemma:projdiff}
  We have $\projdiff(Y, Z) = (\projdiff_1, \projdiff_2)$ where
  \begin{align*}
    \projdiff_1 &= -Y \dot P + \D \qb(X)[\dot X] Y P + \qb(X) Y \dot P - 2 \alpha \D \qc(X)[\dot X] Z P - 2 \alpha \qc(X) Z \dot P\\
    \projdiff_2 &= 2 \sympart\!\Big( -\dot P Y\transpose \qc(X) - P Y\transpose \D \qc(X)[\dot X] + 2 \alpha \dot P Z \qd(X)\\&\qquad\qquad\qquad+ 2 \alpha P Z \D \qd(X)[\dot X] - \dot P Z P - P Z \dot P \Big).
  \end{align*}
\end{lemma}
\begin{proof}
  This is a direct application of the chain rule to the expressions in
  Lemma~\ref{lemma:proj-x-p}.
\end{proof}

We are now ready to express the Riemannian Hessian $\hess \mfc$ as a function of $\grad \sfc$ and $\hess \sfc$.

\begin{proof}[Proof of Proposition~\ref{prop:hessian}]
  Since $\desing$ is a Riemannian submanifold of $\embeddingspace$,
  see for example \citep[(5.34)]{boumal2020introduction} to confirm that
  the Hessian satisfies $\hess \mfc(X, P)[\dot X, \dot P] = w_h + w_g$ with
  \begin{align*}
    w_h = \proj_{(X, P)} \! \big( \hess \sfc(X)[\dot X], \zeros \big) && \text{and} && w_g = \proj_{(X, P)} \! \big( \projdiff(\grad \sfc(X), \zeros) \big)
  \end{align*}
  where $\projdiff$ is defined in~\eqref{eq:def-projdiff}.
  We treat each term above separately.
  From Proposition~\ref{prop:proj} we find that $w_h$ has representation $(\bar
  K_h, \bar V_{p,h})$ given by
  \begin{align*}
    \bar K_h = \hess \sfc(X)[\dot X] V && \text{and} && \bar V_{p,h} = P \hess \sfc(X)[\dot X]\transpose U \Sigma \sfactor(\alpha)^{-1}.
  \end{align*}
  We now consider $w_g$.
  From Lemma~\ref{lemma:projdiff}, we find that $\projdiff(\grad \sfc(X),
  \zeros) = (Y, Z)$ where
  \begin{align*}
    Y &= -\grad \sfc(X) \dot P + \D \qb(X)[\dot X] \grad \sfc(X) P + \qb(X) \grad \sfc(X) \dot P\\
    \text{and } Z &= 2 \sympart\!\pig(-\dot P \grad \sfc(X)\transpose \qc(X) - P \grad \sfc(X)\transpose \D \qc(X)[\dot X]\pig).
  \end{align*}
  It follows from Proposition~\ref{prop:proj} that $w_g$ is a tangent vector
  with representation $(\bar K_g, \bar V_{p, g})$ where $\bar K_g = Y V$ and
  $\bar V_{p, g} = P \big(Y\transpose U \Sigma - 2 \alpha Z V \big)
  \sfactor(\alpha)^{-1}$.
  Using $\dot P V = -V_p$, we compute that
  \begin{align*}
    \bar K_g = \big(I - \qb(X)\big) \grad \sfc(X) V_p = M \grad \sfc(X) V_p
  \end{align*}
  because $I - \qb(X) = M$.
  We now compute $\bar V_{p, g}$.
  First notice that
  \begin{align*}
    P Z V = P \pig(-\dot P \grad \sfc(X)\transpose \qc(X) - P \grad \sfc(X)\transpose \D \qc(X)[\dot X]\pig) V
  \end{align*}
  because $P V = \zeros$ and $P \qc(X)^\top = \zeros$, see~\eqref{eq:qs-usv}.
  So we obtain $\bar V_{p, g} = P \pig( \dot P \grad \sfc(X)\transpose R + \grad
  \sfc(X)\transpose T \pig) \sfactor(\alpha)^{-1}$ where
  \begin{align*}
    R = (\qb(X) - I) U \Sigma + 2 \alpha \qc(X) V && \text{and} && T = \D \qb(X)[\dot X] U \Sigma + 2 \alpha \D \qc(X)[\dot X] V.
  \end{align*}
  From~\eqref{eq:sfactor} and~\eqref{eq:qs-usv} we find that $R = \zeros$.
  From~\eqref{eq:qs} we compute the following differentials:
  \begin{equation}\label{eq:dqs}
    \begin{aligned}
      &\D \qa(X)[\dot X] = -\qa(X)(\dot X X^\top + X \dot X^\top)\qa(X),\\
      &\D \qb(X)[\dot X] = -2 \alpha \D \qa(X)[\dot X],\\
      &\D \qc(X)[\dot X] = \D \qa(X)[\dot X] X + \qa(X) \dot X.
    \end{aligned}
  \end{equation}
  In particular $\D \qc(X)[\dot X] V = \D \qa(X)[\dot X] U \Sigma + \qa(X) K$.
  This implies $T = 2 \alpha \qa(X) K = M K$ because $2 \alpha \qa(X) = M$.
  We finally combine $\bar K = \bar K_h + \bar K_g$ and $\bar V_p = \bar V_{p,
    h} + \bar V_{p, g}$ to obtain the expression for the Hessian.
\end{proof}

\section{Second-order retractions}\label{sec:second-order-retractions}

The lemma below characterizes when the retractions given in
Proposition~\ref{prop:retr-iff} are second-order (see
Definition~\ref{def:second-order-retr}).

\begin{lemma}\label{lemma:cond-zero-acceleration}
  Let $(X, P) \in \desing$ and $(\dot X, \dot P) \in \tangent_{(X, P)}
  \desing$ be represented by $(U, \Sigma, V)$ and $(K, V_\perp L)$.
  Let $t \mapsto P(t)$ be a smooth curve on $\grassmann(n, n - r)$
  satisfying $P(0) = P$ and $P'(0) = \dot P$,
  and let $X(t) = (X + t\dot X)(I - P(t))$.
  Then $c(t) = (X(t), P(t))$ is a smooth curve on $\desing$
  and $c(0) = (X, P)$, $c'(0) = (\dot X, \dot P)$.
  Decompose the initial extrinsic acceleration of $P$ as
  \begin{align*}
    \acc{P}(0) = V A V^\top + V_\perp B V^\top + V B^\top V_\perp^\top + V_\perp C V_\perp^\top.
  \end{align*}
  Then $c$ has zero initial intrinsic acceleration ($\iacc{c}(0) = 0$) if and only if
  \begin{align*}
    \Sigma A = 2 \Sigma L^\top L && \text{and} && B = 2L K^\top U \Sigma \sfactor(\alpha)^{-1}.
  \end{align*}
\end{lemma}
\begin{proof}
  The extrinsic acceleration of $X$ is
  $\acc{X}(t) = -2\dot X P'(t) - (X + t\dot X) \acc{P}(t)$
  so that, with $V_p = V_\perp L$,
  \begin{equation}\label{eq:fgdfgdfg}
    \begin{aligned}
      \acc{X}(0) &= 2 \big( K V^\top + U \Sigma V_p^\top \big) \big(V_p V^\top + V V_p^\top \big) - U \Sigma V^\top \acc{P}(0)\\
                 &= 2 K V_p^\top + 2 U \Sigma V_p^\top V_p V^\top - U \Sigma \big(A V^\top + B^\top V_\perp^\top\big).
    \end{aligned}
  \end{equation}
  The initial intrinsic acceleration is $\iacc{c}(0) = \proj_{(X,
    P)}(\acc{X}(0), \acc{P}(0))$~\cite[\S5.8]{boumal2020introduction}.
  Using Proposition~\ref{prop:proj}, this is a tangent vector whose
  representation $(\bar K, \bar V_p)$ satisfies
  \begin{align*}
    \bar K = \acc{X}(0) V && \text{and} && \bar V_p = P \Big(\acc{X}(0)^\top U \Sigma - 2\alpha \acc{P}(0) V \Big) \sfactor(\alpha)^{-1}.
  \end{align*}
  From~\eqref{eq:fgdfgdfg} we deduce that $\bar K = 2 U \Sigma V_p^\top V_p - U
  \Sigma A$.
  We also compute that
  \begin{align*}
    \bar V_p = V_\perp \pig( 2 L K^\top U \Sigma - B \Sigma^2 - 2 \alpha B \pig) \sfactor(\alpha)^{-1} = 2 V_\perp L K^\top U \Sigma \sfactor(\alpha)^{-1} - V_\perp B.
  \end{align*}
  The conditions $\bar K = \zeros$ and $\bar V_p = \zeros$ give the stated
  equalities.
\end{proof}

\begin{proof}[Proof of Proposition~\ref{prop:global-retr}]
  We let $Z(t) = V + tV_p + \frac{t^2}{2}Y$ where $Y = -2 V_pK^\top U \Sigma
  \sfactor(\alpha)^{-1}$.
  Define the curves $V \colon \reals \to \stiefel(n, r)$ and $P \colon \reals \to
  \grassmann(n, n - r)$ as
  \begin{align*}
    V(t) = Z(t)(Z(t)^\top Z(t))^{-1/2} && \text{and} && P(t) = I - V(t)V(t)^\top.
  \end{align*}
  Then $c(t) = \big((X + t\dot X)(I - P(t)), P(t)\big)$ is a smooth curve on
  $\desing$ with $c(0) = (X, P)$.
  Proposition~\ref{prop:retr-iff} readily gives that $c'(0) = (\dot X, \dot P)$
  so $\retr$ is a retraction.
  We now prove that it satisfies $\iacc{c}(0) = 0$ (zero initial intrinsic
  acceleration) to confirm that $\retr$ is a second-order retraction.

  For all $t$ the matrix $Z(t)$ has rank $r$ so $M(t) = (Z(t)^\top Z(t))^{-1/2}$
  is well defined and satisfies $V(t) = Z(t)M(t)$.
  Differentiating the equality $M(t)^2 = (Z(t)^\top Z(t))^{-1}$ gives
  \begin{align}\label{eq:diffq}
    M'(t)M(t) + M(t)M'(t) = -M(t)^2\pig( Z'(t)^\top Z(t) + Z(t)^\top Z'(t) \pig) M(t)^2.
  \end{align}
  We deduce that $M'(0) = \zeros$ because $M(0) = I$ and $Z'(0)^\top Z(0) =
  \zeros$.
  Now notice that $\acc{Z}(0)^\top Z(0) = Y^\top V = \zeros$.
  So differentiating~\eqref{eq:diffq} again and evaluating at $t = 0$ yields
  \begin{align*}
    \acc{M}(0) = -Z'(0)^\top Z'(0) = - V_p^\top V_p.
  \end{align*}
  We can now evaluate the derivatives of $c$.
  We find that
  \begin{align*}
    V'(t) = (V_p + tY)M(t) + Z(t)M'(t) && \text{and} && P'(t) = -V'(t)V(t)^\top - V(t)V'(t)^\top.
  \end{align*}
  In particular $V'(0) = V_p$.
  Additionally we have $\acc{V}(0) = Y + V \acc{M}(0)$ and
  \begin{align*}
    \acc{P}(0) &= -\acc{V}(0) V^\top - 2V'(0)V'(0)^\top - V \acc{V}(0)^\top\\
               &= - 2 V \acc{M}(0) V^\top -Y V^\top - V Y^\top - 2 V_p V_p^\top.
  \end{align*}
  From Lemma~\ref{lemma:cond-zero-acceleration} we deduce that $c$ has
  zero initial intrinsic acceleration if and only if $2 \Sigma \acc{M}(0) = -2
  \Sigma V_p^\top V_p$ and $Y = -2V_pK^\top U \Sigma \sfactor(\alpha)^{-1}$.
  Both these conditions hold because we computed above that $\acc{M}(0) =
  -V_p^\top V_p^{}$ and because of the definition of $Y$.
\end{proof}

\end{document}